\documentclass[12pt]{amsart}
\usepackage{amsmath, amssymb, amsthm, mathrsfs, textcomp, colonequals,
  comment,soul}
\usepackage{epic}
\usepackage[shortlabels]{enumitem}
\usepackage[all]{xy}
\usepackage{graphicx, subcaption}
\graphicspath{ {./images/} }
\usepackage[backref]{hyperref}
\usepackage[alphabetic,backrefs,lite]{amsrefs}
\usepackage[usenames,dvipsnames]{color}
\usepackage{tikz}
\usepackage{tikz-cd}
\usepackage{fullpage}

\hypersetup{
  colorlinks   = true, 
  urlcolor     = blue, 
  linkcolor    = blue, 
  citecolor   = blue 
}

\numberwithin{equation}{section}
\allowdisplaybreaks[1]
\newtheorem{theorem}[equation]{Theorem}

\newtheorem{corollary}[equation]{Corollary}
\newtheorem{prop}[equation]{Proposition}
\newtheorem{lemma}[equation]{Lemma}

\theoremstyle{definition}

\newtheorem{remark}[equation]{Remark}
\newtheorem{notation}[equation]{Notation}
\newtheorem{defn}[equation]{Definition}

\newtheorem{example}[equation]{Example}

\setenumerate{itemsep=2pt}

\newcommand{\daniele}[1]{{\color{Red} \sf $\clubsuit\clubsuit\clubsuit$ Daniele: [#1]}}

\newcommand{\ol}[1]{\overline{#1}}
\newcommand{\mc}[1]{\mathcal{#1}}
\newcommand{\mf}[1]{\mathfrak{#1}}

\newcommand{\R}{\mathbb R}

\newcommand{\Q}{\mathbb Q}

\newcommand{\Z}{\mathbb Z}
\newcommand{\ints}{\mathbb Z}
\newcommand{\N}{\mathbb N}

\newcommand{\aff}{\mathbb A}

\renewcommand{\P}{\mathbb P}
\newcommand{\proj}{\mathbb P}

\renewcommand{\phi}{\varphi}

\newcommand{\Spec}{\mathrm{Spec}  }
\newcommand{\Frac}{\mathrm{Frac}  }
\newcommand{\frakY}{\mathfrak Y}
\newcommand{\frakX}{\mathfrak X}
\newcommand{\calL}{\mathcal L}
\newcommand{\calB}{\mathcal B}
\newcommand{\calM}{\mathcal M}
\newcommand{\calP}{\mathcal P}
\newcommand{\calO}{\mathcal O}

\newcommand{\panL}{\P^{1, \mathrm{an}}_L}
\newcommand{\scrH}{\mathscr H}

\DeclareMathOperator{\Jac}{Jac}
\DeclareMathOperator{\Gal}{Gal}
\DeclareMathOperator{\Spf}{Spf}

\DeclareMathOperator{\reg}{reg}
\DeclareMathOperator{\stab}{st}

\DeclareMathOperator{\an}{an}
\DeclareMathOperator{\red}{red}

\DeclareMathOperator{\Aut}{Aut}

\DeclareMathOperator{\lcm}{lcm}

\DeclareMathOperator{\Proj}{Proj}
\DeclareMathOperator{\PGL}{PGL}
\DeclareMathOperator{\ord}{ord}

\title{Stabilization indices of potentially Mumford curves}

\author{Andrew Obus}
\address{Baruch College (CUNY)}
\curraddr{1 Bernard Baruch Way. New York, NY 10010, USA}
\email{andrewobus@gmail.com}
\thanks{The first author was supported by NSF grants DMS-1900396 and
  DMS-2047638 as well as
  a grant from the Simons Foundation (\#706038, AO). The second author was supported by the EPSRC New Horizons Grant EP/V047299/1.}

\author{Daniele Turchetti}
\address{Durham University}
\curraddr{Stockton Road, Durham DH1 3LE, United Kingdom}
\email{daniele.turchetti@durham.ac.uk}


\date{\today}

\begin{document}

\begin{abstract}
  Let $X$ be a smooth projective curve over a complete discretely
  valued field $K$.  Let $L/K$ be the minimal extension such that $X
  \times_K L$
  has a semi-stable model, and write $e(L/K)$ for the ramification
  index of $L/K$.  Let $e(X)$ be the so-called
  ``stabilization index'' of $X$, defined by Halle and Nicaise as the $\lcm$ of the
  multiplicities of the ``principal'' irreducible components of a minimal regular snc-model
  of $X$.

  It is known that if $L/K$ is tame, then $e(X) = e(L/K)$.  If one
  drops the tameness assumption, but instead assumes that $X$ has index one and potentially
  multiplicative reduction, 
  Halle and Nicaise ask if the equality $e(X) =
  e(L/K)$ still holds.  We prove that $e(X)$ \emph{divides} $e(L/K)$ in this
  situation, but we give examples, in every residue characteristic, of
  $X$ with $K$-rational points and potentially multiplicative
  reduction such that $e(X) \neq e(L/K)$.
\end{abstract}

\maketitle

\tableofcontents

\section{Introduction}\label{Sintro}

Let $K$ be a complete discretely
  valued field with valuation $v_K$, uniformizer $\pi_K$, valuation ring $\mc{O}_K$ and algebraically
  closed residue field $k$.  This paper investigates the extent to which the
minimal regular model of a smooth projective curve $X$ over $K$ sheds light on the minimal
extension $L/K$ necessary for $X$ to admit a semi-stable model.  We answer a question of Halle and Nicaise (\cite[Question 10.1.1]{HalleNicaise}) in the
negative, showing that, even when $X$ has potentially multiplicative
reduction and $K$-rational points, the so-called \emph{stabilization index} of $X$
need not be equal to $[L:K]$.  However, we show that the stabilization
index of $X$ does
divide $[L:K]$ whenever $X$ has potentially multiplicative reduction.

Let us start by defining the terms above.  Let $X/K$ be a smooth, projective,
geometrically connected curve.
A \emph{normal model} $\mc{X}$ of $X$ is a flat,
normal, proper $\mc{O}_K$-curve with generic fiber isomorphic to $X$.
Each irreducible component $\ol{V}$ of the special fiber of $\mc{X}$
has a multiplicity, given by the degree of the divisor of $\pi_K$ on
$\ol{V}$. 

A normal model $\mc{X}$ of $X$ is called \emph{regular} if it is a
regular scheme.  If the special fiber of a regular model $\mc{X}$ has strict normal
crossings, then $\mc{X}$ is called a \emph{regular snc-model}.  It is
well-known that $X$ possesses a regular snc-model, and that there is a
unique minimal such model if $g(X) \geq 1$ (\cite[Lemma~10.1.8]{LiuBook}).

Suppose that $\mc{X}$ is a regular snc-model of $X$ with special fiber
$\mc{X}_k$.  An irreducible component $\ol{V}$ of $\mc{X}_k$ is called
\emph{principal} if $\ol{V}^{\red}$ has genus at least $1$ or if $\ol{V}$
contains at least $3$ singular points of $\mc{X}_k^{\red}$.  The
\emph{stability index} $e(\mc{X})$ is the lcm of the multiplicities of
the principal components.   We define $e(X) = e(\mc{X})$, where
$\mc{X}$ is any minimal regular snc-model of $X$ (as was mentioned above, there is
only one such model when $g(X) \geq 1$).

A normal model $\mc{X}$ of $X$ 
is called \emph{semi-stable} if its special fiber $\mc{X}_k$ is reduced and has
only ordinary double points for singularities.  It is well-known that,
if $g(X) \geq 2$ or $X$ is an elliptic curve, then there exists a minimal finite
extension $L/K$ such $X_L \colonequals X \times_K L$ has a semi-stable
model, and
furthermore that $L/K$ is Galois.  If $g(X) \geq 2$, there is a
minimal such model $\mc{X}_{\mc{O}_L}^{\stab}$ called the \emph{stable
  model} (\cite[Corollary 2.7]{DeligneMumford}).  The group $\Gal(L/K)$ acts on
$\mc{X}^{\stab}$, and this action is faithful on the
special fiber, see, e.g., \cite[Proposition 2.2.2]{Raynaud}.

\begin{remark}\label{Rfaithful}
 Note that if $L'/K$ is a larger Galois extension, then
$\Gal(L'/K)$ does \emph{not} act faithfully on the special fiber.
More specifically, the action of $\Gal(L'/L)$ is trivial because each
smooth point on the special fiber $\mc{X}_k$ is a specialization of an
$L$-rational point by Hensel's lemma, and the smooth points are dense
in $\mc{X}_k$ since the model is
semi-stable, in particular reduced.
\end{remark}

If $X$ admits a semi-stable model $\mc{X}$ where every irreducible component of the
special fiber $\mc{X}_k$ has genus zero, then $X$ is said to have
\emph{multiplicative reduction}, or to be a \emph{Mumford curve}.  If the same is true after a base
change $L/K$, then $X$ is said to have \emph{potentially
  multiplicative reduction}.

If $g(X) \geq 2$ or $X$ is an
elliptic curve, and if the minimal extension $L/K$ over which $X
\times_K L$ admits a semi-stable model is \emph{tame}, then by
\cite[Proposition~4.2.2.4]{HalleNicaise}, $[L:K] = e(X)$.
In general, however,
$[L:K]$ need not equal $e(X)$.  For instance
\cite[Examples 4.2.2.5, 4.2.2.6]{HalleNicaise} (see also \cite[p.\ 46,
Footnote 5]{Lorenzini}), show that one can have
$([L:K], e(X)) = (2,6)$ or $(6,2)$. Both
examples assume $k$ has characteristic $2$.  In the same
book, Halle and Nicaise ask (\cite[Question 10.1.1]{HalleNicaise})
whether $X$ having potentially multiplicative reduction and index one
(i.e., having a rational point over two different extensions of $K$ of
relatively prime degree)
is sufficient to guarantee that $[L:K] = e(X)$. Our first theorem answers
this question in the negative.

\begin{theorem}\label{Tnegative}
For every prime $p$, there exists $K$ as above with characteristic $p$
residue field, and a smooth, projective, geometrically connected
$K$-curve $X$ with potentially multiplicative reduction and a
$K$-rational point, such that $e(X) \neq [L:K]$, where $L/K$ is the
minimal extension such that $X \times_K L$ admits a stable model.
\end{theorem}

Our second theorem shows that $e(X)$ and $[L:K]$ are still
related, even without the index one assumption.

\begin{theorem}\label{Tmain}
If $X/K$ is a smooth, projective geometrically connected $K$-curve of
genus $g \geq 1$ with potentially
multiplicative reduction, and $L/K$ is defined as in Theorem~\ref{Tnegative}, then $e(X) \mid [L:K]$.
\end{theorem}

\begin{remark}\label{Rwrongdivisibility}
By \cite[Example 4.2.2.5]{HalleNicaise} mentioned above,
Theorem~\ref{Tmain} does \emph{not} hold without assuming potentially
multiplicative reduction.
\end{remark}

\subsection{Motivation}\label{Smotivation}
The immediate motivation for our work is the question of Halle and
Nicaise mentioned above.  Let us place this in
context --- see \cite[\S10.1]{HalleNicaise} for more details. Recall that $K$ is a complete discretely valued field with
algebraically closed residue field $k$.  Suppose $A/K$ is an abelian
variety.  A central motivation of \cite{HalleNicaise} is to define a
``stabilization index'' $e(A)$ in a natural way.  The stabilization
index $e(A)$ should be an integer such that
for all $d$ prime to both $e(A)$ and char$(k)$, both the N\'{e}ron
component group $\Phi(A)$ and the motivic class of the identity
component of the N\'{e}ron model of $A$ behave well under base change by the unique
tame extension $K(d)/K$ in a precise sense, leading to 
rationality results for certain motivic zeta functions.  If $A$ acquires semi-stable reduction over a tamely ramified extension, or if $A$ acquires multiplicative
reduction over any extension, then setting $e(A) = [L:K]$
works, where $L/K$ is the minimal such extension.  Alternatively, if $A = \Jac(X)$ where $X$ is a smooth
projective curve with index one, then
setting $e(A)$ equal to the stabilization index $e(X)$ works.  As we
have mentioned, if $X$ has semi-stable reduction over a tame extension
$L/K$ (and thus $\Jac(X)$ does as well), then $e(X) = [L:K]$, so both
definitions of $e(A)$ coincide for $A = \Jac(X)$.  This leads one to ask if they also
coincide when $A = \Jac(X)$ for $X$ an index one curve with potentially multiplicative
reduction, but not necessarily realized over a tame extension of $K$.
Theorem~\ref{Tnegative} shows that this is not true in general.

Alternatively, for an arbitrary abelian variety $A$, one can try to define $e(A)$ to
be the lcm of denominators of the ``jumps'' of $A$.  These jumps are certain
rational numbers
encoding how the Lie algebras of the N\'{e}ron models of various base changes of $A$
fit together, see \cite[\S6]{HalleNicaise}.  If $e(A)$ is defined this
way, Halle and Nicaise show that $e(A) = e(X)$ when $A = \Jac(X)$
for \emph{any} smooth projective $K$-curve $X$.
Combining this with Theorem~\ref{Tnegative}, we see that there exists a curve $X_0$ with
potentially multiplicative reduction such that
if $A = \Jac(X_0)$, then $e(A) \neq [L:K]$, where $L/K$ is the minimal
extension over which $A$ (equivalently $X_0$) has
semi-stable reduction.

In fact, the jumps make sense for semi-abelian
varieties.  So letting $T$ be the torus uniformizing $A =  \Jac(X_0)$,
we have that $e(T) = e(A)$ where $e(T)$ is the lcm of the denominators
of the jumps for $T$, since the jumps depend only on Lie
algebras.  So $e(T) \neq
[L:K]$.  This answers
\cite[Question 10.1.2]{HalleNicaise} negatively.  We mention that Overkamp
(\cite[Corollary 2.2.2 and its remark]{Overkamp}) also constructs a torus answering
\cite[Question 10.1.2]{HalleNicaise} negatively.  His construction
uses completely different methods, and does not involve Jacobians of curves.

%
%
At a more basic level, this paper continues the study
originated by Saito in
\cite{Saito} and continued, e.g., in 
\cite{Lorenzini}, \cite{Halle}, and \cite{Nicaise}, of what one can say about the
relationship between the minimal extension $L/K$ over which a curve
$X$ attains stable reduction, the stable model $\mc{X}^{\stab}$
 of $X \times_K L$, and the minimal
regular snc-model $\mc{X}$ of $X$.  Let $p = \text{char}(k)$.
Saito originally showed that $p \nmid [L:K]$\footnote{equivalent to condition (1) of
\cite[Theorem 3]{Saito} by \cite[Theorem 1]{Saito} and \cite[Theorem 1.6]{Abbes}} 
if and only if $p \nmid e(X)$\footnote{equivalent to condition (2) of
\cite[Theorem 3]{Saito}}, in which case we have already seen that
$[L:K] = e(X)$.
In \cite[Question 1.4]{Lorenzini}, Lorenzini asked about a natural
generalization of Saito's criterion.  Namely, if char$(k) = p$, then do we have
$v_p(e(X)) \leq v_p([L:K])$?  Theorem~\ref{Tmain} answers this
question positively in the case of potentially multiplicative
reduction, and our examples also show that
the inequality can be strict.


\subsection{Techniques and outline}\label{Stechniques}

One strategy to establish a link between regular models and stable models
is as follows:  If $X$ is a smooth, projective, geometrically
connected $K$-curve with stable model $\mc{X}^{\stab}$ over $L$, then the resulting
quotient curve $\mc{X}^{\stab} / \Gal(L/K)$ is a normal
model of $X$ with quotient singularities.  So obtaining information
about $\mc{X}$ from $L/K$ and $\mc{X}^{\stab}$ is
tantamount to resolving quotient singularities.  Our proofs of
Theorems~\ref{Tnegative} and \ref{Tmain} rely on explicitly resolving
such singularities when $X \times_K L$ has multiplicative reduction.

Throughout the paper, we exploit Mumford's result that curves with
potentially multiplicative reduction, when viewed as analytic spaces,
can be realized as quotients of subsets of the projective line by
Schottky groups; the details of this story are reviewed in
\S\ref{Smumford}.  Indeed, all of our examples for
Theorem~\ref{Tnegative} are presented by specifying the relevant
Schottky group explicitly.
Mumford's perspective is useful because it ultimately allows us to
view the quotient singularities arising from stable models as lying on certain models of $\proj^1_K$, rather than
on higher genus curves.\footnote{A good example of this process in
  action is the proof of Proposition~\ref{Pexceptionalmultiplicity1}.}  This is useful because on models of $\proj^1_K$, explicit
resolution of singularities is well-understood, see in particular
\cite[Theorem~7.8]{ObusWewers} by Wewers and the first author.

The main consequence of \cite[Theorem~7.8]{ObusWewers} that we use in this
paper is Proposition~\ref{Pmaclaneresolution}, which states that,
given a normal model $\mc{X}$ of $\proj^1_K$ with special fiber $\mc{X}_k$, every principal
component of the special fiber of the minimal regular snc-resolution of $\mc{X}$ has
multiplicity dividing the lcm of the multiplicities of the components
of $\mc{X}_k$.  That is, the exceptional divisor introduces no
principal components with ``bad'' multiplicity.
The proof of Proposition~\ref{Pmaclaneresolution} (as well as that of
\cite[Theorem~7.8]{ObusWewers}) uses ``Mac Lane
valuations'', which are a convenient notation for expressing
so-called ``geometric valuations'' of $K(x)$, which correspond to
irreducible components of normal models of $\proj^1_K$.  In short, we
can explicitly write down the valuations corresponding to the
irreducible components of the exceptional divisor, and it is fairly
straightforward to see that the desired divisibilities of
multiplicities hold.

To go from Mumford's (analytic) perspective on curves with potentially
multiplicative reduction to explicit resolution of
singularities on (algebraic) models of the projective line, we
frequently pass between Berkovich-analytic, rigid-analytic,
formal, valuation-theoretic, and algebraic perspectives on curves.\footnote{It
is perhaps unfortunate that we have to involve all of these perspectives,
but since our results fundamentally depend on references to the theory
of Mac Lane valuations, we ultimately need to express everything in
terms of standard algebraic models of curves to avoid re-working that
theory.} 
The important ``dictionaries'' here are discussed in \S\ref{Smodels}.
In \S\ref{Smumford}, we give a detailed description of those aspects
of the theory of Mumford curves we use in the paper from the
Berkovich-analytic perspective.  In \S\ref{Smaclane}, we give a brief
introduction to Mac Lane valuations, and prove
Proposition~\ref{Pmaclaneresolution}.  

In
\S\ref{Sdescent}, if $L/K$ is a Galois extension, we give a sufficient
criterion for a Mumford curve over 
$L$ to be defined as a curve (not necessarily a Mumford curve) over
$K$.  This is of independent interest, and is perhaps well-known, but
we could not find a suitable reference.  In any case, it is
necessary to ensure that our examples, which are presented as Mumford
curves over $L$, in fact descend to curves over $K$. In \S\ref{Sexample}, we present our examples,
proving Theorem~\ref{Tnegative}, and in \S\ref{Sdivisibility}, we
prove Theorem~\ref{Tmain}.

  \begin{notation}
  Throughout the paper, we adopt the following conventions:
  we denote by $K$ a complete discretely valued field with valuation
  $v_K$, valuation ring $\mc{O}_K$, uniformizer $\pi_K$, and algebraically
  closed residue field $k$.  The valuation $v_K$ is normalized so that
  $v_K(K^{\times}) = \ints$.
  By a $K$-curve we mean a smooth, projective, geometrically connected algebraic curve over $K$.
  If $X$ is such a $K$-curve, we denote by $X^{\mathrm{an}}$ its Berkovich analytification.
\end{notation}

\subsection*{Acknowledgements}
We thank Lars Halvard Halle, Johannes Nicaise, and Lorenzo Fantini for
insightful conversations, as well as J\'er\^ome Poineau for useful
comments on a previous version of this text.  Lastly, we thank the
referee for useful comments.

\section{Preliminaries}\label{Sprelims}
\subsection{Models of curves and analytic geometry}\label{Smodels}

Let $X$ be a $K$-curve. 
To $X$, we associate its Berkovich analytification $X^{\mathrm{an}}$ that we briefly define below.
Note that in this process we always work on $K$ and do not need to  perform any base-change to its algebraic closure.
As a set, $X^{\mathrm{an}}$ consists of pairs $(\xi, |\cdot|_\xi)$
where $\xi$ is a point of $X$ and $|\cdot|_\xi: \kappa(\xi) \to
\R_{\geq 0}$ is an absolute value on the residue field of $\xi$
extending the absolute value of $K$.

The set $X^{\mathrm{an}}$ can be endowed with the structure of a locally ringed space that makes it a $K$-analytic space in the sense of Berkovich's theory \cite{Berkovich90}. 
In analogy with scheme theory, Berkovich's theory associates with a complete non-archimedean field $K$ its \emph{analytic spectrum} $\mc{M}(K)$. Saying that $\mf{X}$ is a $K$-analytic space is a shorthand for fixing a morphism of analytic spaces $\mf{X} \to \mc{M}(K)$.

The points of the $K$-analytic curve $X^{\mathrm{an}}$ can be classified as follows:
\begin{itemize}
\item If $\xi$ is a closed point of $X$, then the associated residue field $\kappa(\xi)$ is a finite extension of $K$, giving rise to a unique corresponding point in $X^{\mathrm{an}}$.
Every point of $X^{\mathrm{an}}$ arising in this way is called a rigid
point of \emph{type 1}.\footnote{Note that our definition of type 1 is slightly more restrictive than Berkovich's one. This latter includes all the points $x$ whose completed residue field $\scrH(x)$ is contained in the completion of the algebraic closure of $K$.}
\end{itemize}
If $x=(\xi, |\cdot|_\xi)$ is a point of $X^{\mathrm{an}}$ such that $\xi$ is the generic point of $X$, then we can consider the function 
\[ \begin{array}{rcl}
v_x:K(X) & \longrightarrow &\R \cup \{\infty\} \\
f &\longmapsto &- \log|f|_\xi .
\end{array}\]
It is a real valued valuation on the function field of $X$, extending the discrete valuation of $K$.
Conversely, every such valuation gives rise to a point of $X^{\mathrm{an}}$ which is not of type 1.
In this context, we call $\scrH(x)$ the completion of $K(X)$ with respect to $v_x$, we denote by $\widetilde{\scrH(x)}$ its residue field, and we distinguish three different cases:
\begin{itemize}
\item If the transcendence degree of $\widetilde{\scrH(x)}$ over $k$ is equal to 1, then we say that $x$ is of \emph{type 2}.
\item If the dimension of the $\Q$-vector space $\frac{|\scrH(x)^\times|}{|K^\times|} \otimes_\Z \Q$ is equal to 1, then we say that $x$ is of \emph{type 3}.
\item In the remaining cases, we say that $x$ is of \emph{type 4}.
\end{itemize}



\begin{remark}\label{Rgeometric}
If $x \in X^{\an}$ is a type 2 point, then by Abhyankhar's inequality \cite[\S 6, 10.3]{BourbakiAC57} we have that $\frac{|\scrH(x)^\times|}{|K^\times|}$ is a finite group.
Since $|K^\times| \cong \Z$ and $|\scrH(x)^\times|\subset \R_{\geq 0}$ is torsion free, then $|\scrH(x)^\times|$ is also abstractly isomorphic to $\Z$, that is, $x$ is an absolute value induced by a discrete valuation on $\scrH(x)$. Thus we can think of type 2 points
as discrete valuations on $K(X)$ restricting to $v_K$ on $K$
whose residue fields have transcendence degree 1 over $k$.  Such
valuations are called \emph{geometric valuations}.  In
\S\ref{Smaclane}, we recall a useful notation for such valuations
developed by Mac Lane.
\end{remark}

In this paper, we only deal with points of type 1 and type 2.
Points of type 2 are especially important for the following reason: if
$\mc{X}$ is a model of $X$ with special fiber $\mc{X}_k$, reduction modulo $\pi_K$ induces a natural surjective and anti-continuous map
\[
\red_\mc{X}:X^{\mathrm{an}} \longrightarrow \mc{X}_k
\]
(that is, the inverse image of an open subset of $\mc{X}_k$ is closed in $X^{\mathrm{an}}$), called the \emph{specialization map} (or \emph{reduction map}) associated with $\mc{X}$.
If $\mc{X}$ is normal and $\eta$ is the generic point of an irreducible component of $\mc{X}_k$, then $\red_\mc{X}^{-1}(\eta)$ consists of a single type 2 point of $X^{\mathrm{an}}$.
This construction was first investigated in
\cite{BoschLuetkebohmert85} in the setting of rigid analytic geometry
and is at the heart of the description of the structure of
non-archimedean $K$-analytic curves. 
The behavior of $\red^{-1}_{\mc{X}}$ on generic points of $\mc{X}_k$ can be made precise as follows: if $\ol{V}$ is an irreducible component of $\mc{X}_k$ with generic point $\eta_{\ol{V}}$,
then the associated type 2 point $\red^{-1}_{\mc{X}}(\eta_{\ol{V}})$ is
the valuation
\begin{equation}\label{Evaldef}
f \mapsto \frac{1}{m_{\ol{V}}}\ord_{\ol{V}}(f),
\end{equation}
where $m_{\ol{V}}$ is the multiplicity of $\ol{V}$.  

One can show that the map associating with $\mc{X}$ the set of type 2
points corresponding to the irreducible components of $\mc{X}_k$ is
actually a bijective correspondence.
More precisely, we have the following result.

\begin{prop}\label{prop:models}
The correspondence 
	\[
	\mc{X} \longmapsto V_\mc{X} = \big\{\red_\mc{X}^{-1}(\eta)\,\big|\,\eta\mbox{ generic point of a component of }\mc{X}_k\big\},
	\]
induces an isomorphism between the partially ordered set of isomorphism
classes of normal (algebraic) $\mc{O}_K$-models of $X$, ordered by
domination, and the partially ordered set of non-empty finite subsets
of $X^{\an}$ containing only type 2 points (or equivalently via Remark~\ref{Rgeometric}, non-empty finite sets of geometric valuations on $K(X)$), ordered by inclusion.
	
\end{prop}

\begin{proof}
\cite[Corollary~3.18]{Ruth}, and \cite[Theorem~4.1]{FantiniTurchetti21} for the same statement in the context of Berkovich curves.
\end{proof}

\begin{defn}\label{Dabc}
\text{}
\begin{enumerate}
\item Let $x\in X^{\mathrm{an}}$ be a type 2 point. Then, the correspondence
of Proposition \ref{prop:models} assigns to the singleton $\{x\}$ a
model $\mc{X}_1$ with irreducible special fiber $\mc{X}_{1,k}$. We
define the \emph{multiplicity} of $x$, denoted by $m(x)$, to be the multiplicity of $\mc{X}_{1,k}$.
\item If $p$ is a closed point of the special fiber $\mc{X}_k$, its inverse image $\red^{-1}_{\mc{X}}(p)$ is called the \emph{formal fiber} of $p$, and it is an open subset of $X^{\mathrm{an}}$.
As such, it can not be an affinoid domain but enjoys similar properties, belonging to the wider class of \emph{semi-affinoid $k$-analytic spaces}\footnote{this terminology was introduced by Martin and Kappen in \cite{Martin17} to refer to a special class of analytic spaces studied by Berthelot in \cite{Berthelot96}, see also \cite[\S 7]{Jong95}}.
\end{enumerate}
\end{defn}

\begin{remark}
If $S$ is a set of type 2 points of $X^{\mathrm{an}}$ containing $x$, then the corresponding model $\mc{X}_2$ is obtained from $\mc{X}_1$ via a sequence of blow-ups followed by a sequence of blow-downs that do not affect the component $\ol{V}_x$ containing the strict transform of $\mc{X}_{1,k}$. 
In particular, the multiplicity of $\ol{V}_x$ in $\mc{X}_2$ is again
$m(x)$, so that the multiplicity of $x$ can be read off from any model arising from a set $S$ as above.
\end{remark}

\begin{remark}\label{Dc}
 In the specific case of the projective line we have the following:
a type 2 point $x\in \P^{1,\mathrm{an}}_K$ of
multiplicity 1 is a sup-norm on a $K$-rational closed disc.
In other words, there exist $a\in K$ and $\rho \in |K^\times|$ such that $x$ is a valuation of the form 
\[\eta_{a,\rho}(f):=\sup_{z\in \ol{B}(a,\rho)}\{|f(z)|\},\] where $\ol{B}(a,\rho)= \{z \in K: v_K(z - a) \geq \rho\}$.
In fact, there exists a closed point $p$ of $\ol{V}_x$ whose formal fiber is a Berkovich open disc centered in a $K$-point with radius in the value group of $K$. Then, $x$ is the boundary point of such a disc, and it follows from \cite[\S2.5]{Berkovich90} that it coincides with the sup-norm $\eta_{a,\rho}$.
\end{remark}

\begin{defn}\label{def:skeleton}
If $X$ has a semi-stable model $\mc{X}$, then the dual graph of the special fiber $\mc{X}_k$ is canonically represented by a subset $\Sigma_{\mc{X}}$ of the analytification $X^{\mathrm{an}}$ as follows: the set of vertices of $\Sigma_{\mc{X}}$ is the collection of points $\red^{-1}_{\mc{X}}(\xi)$ for $\xi$ varying over the generic points of all the irreducible components of $\mc{X}_k$, while the set of edges is given by those intervals contained in the formal fibers $\red^{-1}_{\mc{X}}(x)$ that join two vertices of $\Sigma_{\mc{X}}$. 
One can see that this happens only when $x$ is a closed point of $\mc{X}_k$ that belongs to two different irreducible components.
There is a continuous map \[\rho_X:X^{\mathrm{an}} \to \Sigma_{\mc{X}}\]
which makes $\Sigma_{\mc{X}}$ into a strong deformation retract of
$X^{\mathrm{an}}$. 
If the genus of $X$ is at least 1, then there is a minimal subset of $X^{\mathrm{an}}$ that supports a graph of the form $\Sigma_{\mc{X}}$ for some model $\mc{X}$.
We denote any such graph by $\Sigma_X$ and we call it the \emph{skeleton}
of $X^{\mathrm{an}}$. By extension we will refer to it as the
skeleton of $X$ as well. If the genus of $X$ is at least 2, then $\Sigma_X$ is the skeleton associated with the stable model of $X$.
\end{defn}

\subsection{Mumford curves}\label{Smumford}
A central tool for proving our main result is the Schottky uniformization of a Mumford curve $X$.
In \S\ref{Suniformization} we start by reviewing Mumford's uniformization theorem for these curves (Theorem \ref{thm:unif}), using the point of view of Berkovich analytification outlined in the previous section.
Among other things, Mumford's theorem establishes the existence of a uniformizing analytic morphism $p:\mf{D} \to X^{\mathrm{an}}$, which is a universal covering at the level of topological spaces.
The idea of applying Berkovich geometry to this framework is classical, being already proposed by Berkovich in \cite[Section~4.4]{Berkovich90}, and related to results established in the context of rigid-analytic geometry, for example by Gerritzen and van der Put \cite{GerritzenPut80} and L\"utkebohmert \cite{Luetkebohmert16}.\\
In \S\ref{Suniformization} we recall some facts following the text \cite{PoineauTurchetti21} that develops the theory of Mumford's uniformization from Berkovich's point of view.
In \S\ref{Sformal} we establish new results on the formal geometry of models of Mumford curves and their uniformizations that will be used later on in the paper.
In Proposition \ref{prop:mult1} we build a regular semi-stable model of $X$ that will be used as a basis for other constructions. 
Then, we explain how given a semi-stable model $\mc{X}'$ of $X$ we can lift the uniformization map $p:\mf{D} \to X^{\mathrm{an}}$ to a morphism of formal models $\varpi:\mc{D}' \to \hat{\mc{X}}'$ (see Lemma \ref{lem:formallift} and the preceding construction).
Finally, we establish properties of this lifting and comparison results between certain local rings of $\mc{D}'$ and of $\mc{X}$. 
This allows us to move freely from the language of Berkovich curves and that of (formal) models in the rest of the paper, getting in this way the best of both worlds.
In \S\ref{Slifting} we consider automorphisms of $X$ that are semi-linear with respect to a Galois sub-extension of the field where $X$ is defined.
In particular, Proposition \ref{prop:extensionofliftings} establishes conditions for these automorphisms to lift to the universal cover of $X$, which will be used at the beginning of \S\ref{Sdescent} and \S\ref{Sdivisibility}.

\subsubsection{Uniformization of Mumford curves}\label{Suniformization}
Let $g\in \N_{\ge 1}$. Recall that a \emph{Mumford curve} over $K$ is an algebraic $K$-curve
whose Jacobian has split totally degenerate reduction.  Equivalently,
it is a $K$-curve with a semi-stable model over $\mc{O}_K$ such that
the irreducible components of the special fiber all have genus zero.
Mumford showed that these are precisely the non-archimedean curves
that admit a uniformization as in
Theorem~\ref{thm:unif} below.  Before stating this theorem, we need
some definitions.

\begin{defn}[cf. Definition 6.4.1 in \cite{PoineauTurchetti21}]\label{def:Schottkyfigure}
Let $\gamma_{1},\dotsc,\gamma_{g} \in \PGL_{2}(K)$. 
Let $\calB = \big(D^+(\gamma_{i}), D^+(\gamma_{i}^{-1})\big)_{1\le i\le g}$ be a set consisting of $2g$ pairwise disjoint closed discs in~$\P_K^{1, \mathrm{an}}$.
For each $\gamma\in\{\gamma_1^{\pm1},\dotsc, \gamma_g^{\pm1}\}$ we define an open disc $D^-(\gamma)$ by setting
\begin{equation}\label{eq:Schottkyfig} D^-(\gamma) := \gamma (\P^1_{K} \setminus D^+(\gamma^{-1})).
\end{equation}
The set $\calB$ is called a \emph{Schottky figure} adapted to $(\gamma_{1},\dotsc,\gamma_{g})$ if, for each $\gamma \in\{\gamma_1^{\pm1},\dotsc, \gamma_g^{\pm1}\}$ we have that $D^-(\gamma)$ is a maximal open disc inside $D^+(\gamma)$.
A subgroup $\Gamma$ of $\PGL_{2}(K)$ is called a \emph{Schottky group} if there exist a generating set $\{\gamma_1, \dotsc, \gamma_g\}$ for $\Gamma$ and a Schottky figure adapted to the $g$-tuple $(\gamma_{1},\dotsc,\gamma_{g})$.
\end{defn}

\begin{figure}
\includegraphics[scale=.5]{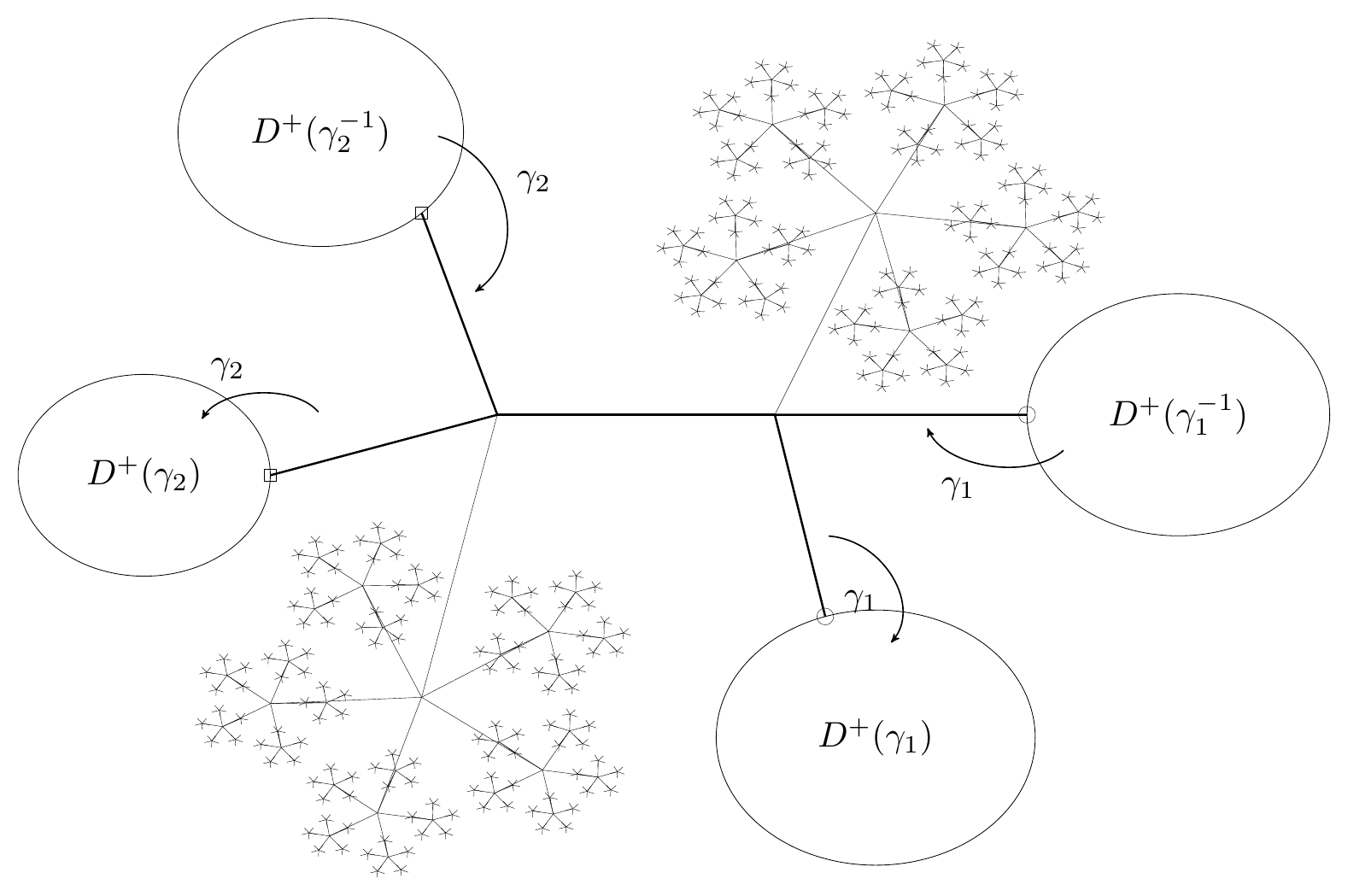}
\caption{A Schottky figure adapted to a pair $(\gamma_1, \gamma_2)$.}
\end{figure}

\begin{notation}\label{NSchottky}
Let $\calB$ be a Schottky figure adapted to a $g$-tuple $(\gamma_{1},\dotsc,\gamma_{g})$.
The \emph{fundamental domain} associated with $\calB$ is
\[F := \P^{1, \mathrm{an}}_{K} \setminus \bigcup_{i=1}^g \big( D^-(\gamma_i) \cup D^{+}(\gamma_i^{-1}) \big).\]
The \emph{discontinuity set} and the \emph{limit set} of the Schottky group $\Gamma$ are respectively
\[\mf{D} :=\bigcup_{\gamma\in\Gamma} \gamma F \;\;\; \mbox{and} \;\;\; \calL:= \P^{1, \mathrm{an}}_{K} \setminus \mf{D}.\]
The \emph{skeleton} $\Sigma_{\mf{D}}$ of $\mf{D}$ is the intersection
of $\mf{D}$ with the convex envelope of~$\calL$:
\[\Sigma_{\mf{D}} := \mf{D} \cap \bigcup_{x,y \in \calL} [x,y].\]
The \emph{skeleton} $\Sigma_{F}$ of $F$ is the subset $F\cap \Sigma_{\mf{D}}$ of $F$.
Note that both $\mf{D}$ and $F$ retract continuously on their skeletons as in Definition \ref{def:skeleton}.
\end{notation}

\begin{remark}
The discontinuity set and the limit set depend only on $\Gamma$, and
not on the choice of a Schottky figure.
Moreover, the limit set coincides with the set of limit points of
orbits under the action of $\Gamma$, that is $\ell \in \P^{1,
  \mathrm{an}}_{K}$ belongs to $\calL$ if and only if there exist
$x\in \P^{1, \mathrm{an}}_{K}$ and a sequence $(\gamma_i)_{i\in \N}$
of elements of $\Gamma$ such that  $\gamma_i(x) \neq \ell$ for all
$i \in \N$ and $\lim_{i\to\infty} \gamma_i(x) =
\ell$.
The points of the limit set are $K$-rational (see \cite[\S4.4]{Berkovich90}).
\end{remark}

The following theorem describes Mumford's uniformization in the
setting of Berkovich geometry.

\begin{theorem}[cf. Theorem 6.4.18 in \cite{PoineauTurchetti21}]\label{thm:unif}
Let $\Gamma$ be a Schottky group with $g$ generators.
The action of~$\Gamma$ on~$\mf{D}$ is free and proper and the quotient $\Gamma\backslash \mf{D}$ is the analytification of a Mumford curve $X^{\mathrm{an}}$ of genus~$g$. 
Conversely, given a Mumford curve $X$ of genus $g$ there exists a Schottky group $\Gamma$ of rank $g$ such that $X^{\mathrm{an}}= \Gamma\backslash \mf{D}$, where $\mf{D}$ is the discontinuity set of $\Gamma$.

The quotient map $p \colon \mf{D} \to X^{\mathrm{an}}$ is a universal covering of $X^{\mathrm{an}}$ respecting the skeleta: if $\Sigma_{\mf{D}}$, $\Sigma_{F}$ and $\Sigma_{X}$ denote the skeleta of~$\mf{D}$, $F$ and~$X$ respectively, we have
\[ p^{-1}(\Sigma_{X}) = \Sigma_{\mf{D}} \textrm{ and } p(\Sigma_{\mf{D}}) = p(\Sigma_{F}) = \Sigma_{X}. \]
\end{theorem}

\begin{figure}[ht]
\centering
    \begin{subfigure}[b]{0.4\textwidth}
\includegraphics[scale=.3]{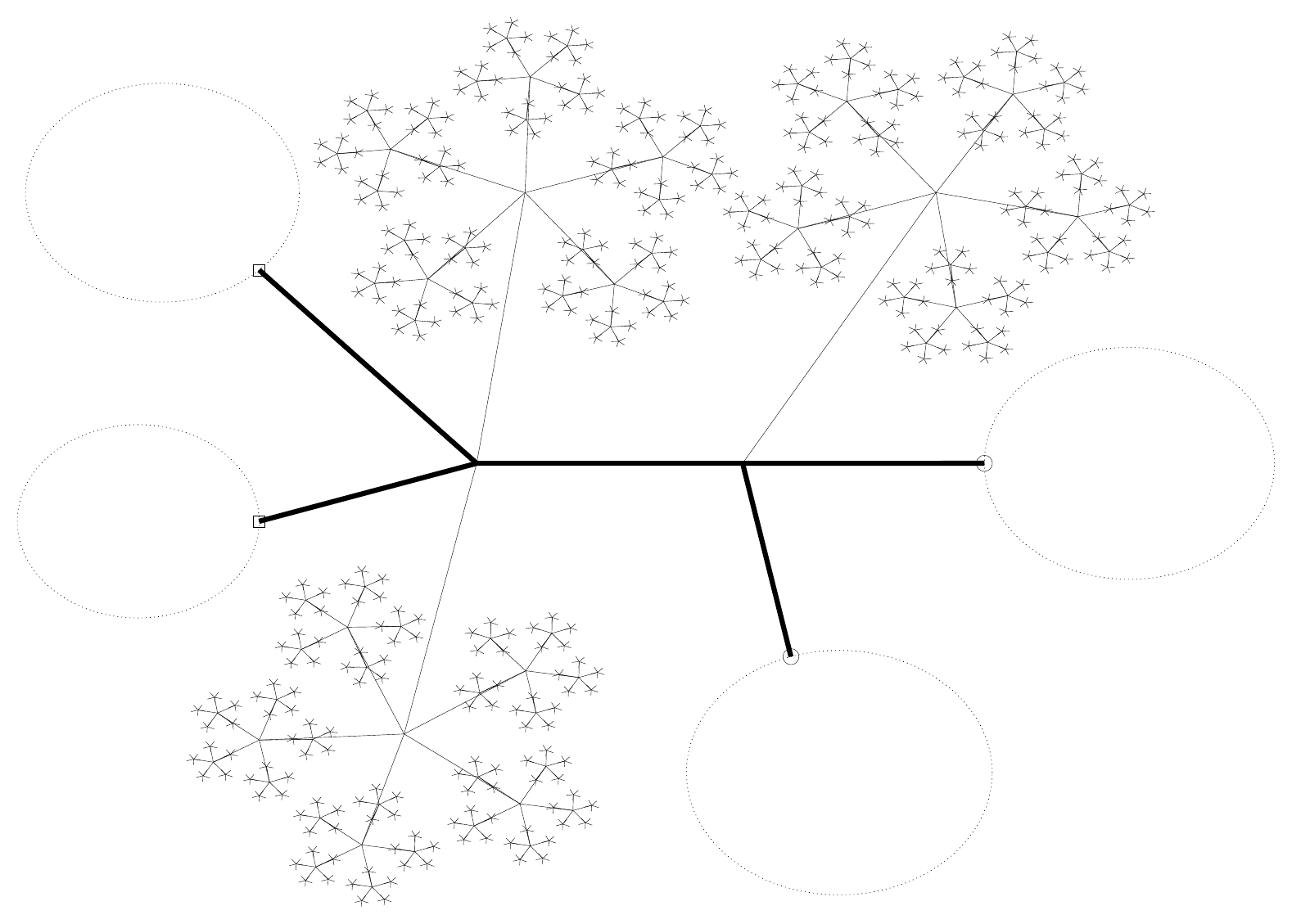}
    \end{subfigure}
    \hspace{1.5cm}
    \begin{subfigure}[b]{0.4\textwidth}
\includegraphics[scale=.3]{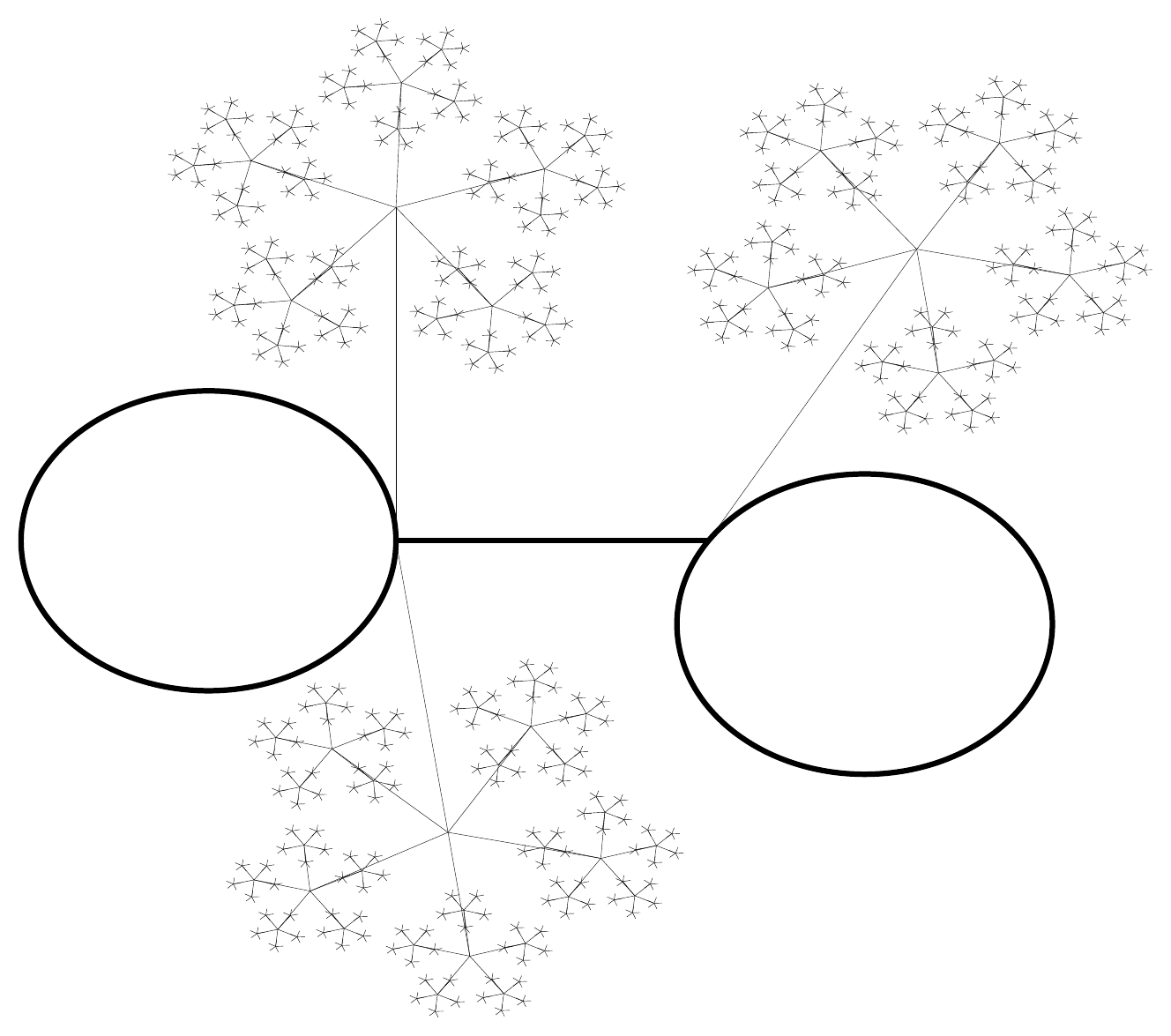}
    \end{subfigure}

 \caption{The fundamental domain $F$ (on the left) of the Schottky group $\Gamma$ is the complement of $2g$ discs in $\P^{1, \mathrm{an}}_K$. The group $\Gamma$ identifies the ends of the skeleton $\Sigma_{F}$, so that the corresponding curve (on the right) retracts on its skeleton $\Sigma_X$.}
 \label{fig:Test}
\end{figure}

\begin{remark}\label{rmk:Mumformal}
Originally, Theorem \ref{thm:unif} was proved by Mumford \cite{Mumford72} using tools of formal algebraic geometry.
The central objects in Mumford's paper are a certain formal scheme $\mc{D}$ whose generic fiber is isomorphic to $\mf{D}$, and a uniformization map $\mc{D}\to \mc{X}$, where $\mc{X}$ is the stable model of $X$.
The relation between these formal schemes and Berkovich analytifications will be established in the second part of the present section.
\end{remark}

\begin{example}[Tate curves]
If $g=1$ in the theory above, one starts with the datum of an element $\gamma \in \PGL_2(K)$ and of two disjoint closed discs $D^+(\gamma)$  and $D^+(\gamma^{-1})$ in such a way that
$\gamma(\P_K^{1, \mathrm{an}} \setminus D^+(\gamma^{-1}))$ is a maximal open disc inside $D^+(\gamma)$.
Since $\gamma$ is loxodromic, up to conjugation it is represented by a matrix of the form $\begin{bmatrix} q & 0 \\ 0 & 1 \end{bmatrix}$ for some $q \in L$ satisfying $0<|q|<1$.
In other words, up to a change of coordinate in $\P^1_K$, the transformation $\gamma$ is the multiplication by $q$ and hence the limit set $\mc{L}$ consists only of the two points $0$ and $\infty$.
The quotient curve obtained from applying Theorem \ref{thm:unif} is an elliptic curve, whose set of $K$-points is isomorphic to the multiplicative group $K^\times/ q^\Z$.
\end{example}

If $g\geq 2$ the skeleton $\Sigma_{X}$ of $X^{\mathrm{an}}$ coincides with the dual graph of the special fiber $\mc{X}_k$ of the stable model of $X$ by \cite[4.3]{Berkovich90}.

\subsubsection{Formal models and uniformization}\label{Sformal}

The uniformization map $p \colon \mf{D} \to X^{\mathrm{an}}$ of Theorem \ref{thm:unif} is a main player of this paper. 
In the following, whenever a Mumford curve is given we always assume that it comes with the choice of a map $p$ as well.
We firstly need to study the relationship between the analytic theory of uniformization and the formal-algebraic theory of models. 
Let us start with a result that is proper to our setting, where $K$ is a discretely valued field:

\begin{prop}\label{prop:mult1}
Let $X$ be a Mumford curve of genus at least 2, or of genus $1$ with
$X(K) \neq \emptyset$ .
Then the set $\calP$ of points of multiplicity 1 contained in the skeleton $\Sigma_{X}$ of $X^{\mathrm{an}}$ is finite.
Moreover, the model associated with $\calP$ under the correspondence given by Proposition \ref{prop:models} is semi-stable.
\end{prop}
\begin{proof}
The curve $X$ has semi-stable reduction by virtue of being a Mumford curve, and we fix a minimal semi-stable
model $\mc{X}$ of $X$ with special fiber $\mc{X}_k$. 
Then, Definition \ref{def:skeleton} ensures that the skeleton $\Sigma_{X}$ is non-empty and that there is a point of multiplicity 1 on $\Sigma_{X}$ for every irreducible component of $\mc{X}_k$.
Let $\widetilde{\mc{X}}$ be the minimal desingularization of
$\mc{X}$.
The model $\widetilde{\mc{X}}$ is the minimal regular snc-model of
$X$, and it is obtained by blowing up $\mc{X}$ at the closed double
points that are singular in $\mc{X}_k$.
Let $p$ be such a double point. Its formal fiber $\red^{-1}_\mc{X}(p)$ is an open annulus of thickness $n$ contained in $X^{\mathrm{an}}$.
Then the model $\mc{X}$ needs to be blown-up exactly $\lfloor n/2\rfloor$ times at the point $p$ in order to resolve the singularity (see
\cite[Example 8.3.53]{LiuBook}).
This procedure corresponds to changing a model by adding all the multiplicity 1 points in $\Sigma_X \cap \red^{-1}_\mc{X}(p)$ to the corresponding set of type 2 points.
As a result, $\widetilde{\mc{X}}$ is semi-stable, and we have a bijection between the irreducible components of $\widetilde{\mc{X}}_k$ and the points of multiplicity 1 in $\Sigma_{X}$.
In other words, the semi-stable model $\widetilde{\mc{X}}$ is the one associated with the set $\calP$.
\end{proof}

We now proceed to describe Mumford's construction of the formal model $\mc{D}$ of $\mf{D}$ (see Remark \ref{rmk:Mumformal}) using Bosch and L\"utkebohmert's theory of \emph{canonical reductions} (see \cite[\S 1]{BoschLuetkebohmert85} for full details on this theory).
First, we recall that an affinoid space $U$ has an associated
canonical affine formal model, which is the formal spectrum $\mc{U}=\Spf(\calO_U(U)^\circ)$ where $\calO_U(U)^\circ \subset \calO_U(U)$ is the subring of functions bounded by 1 on $U$.
This association gives rise to a canonical reduction map $\red_\mc{U}:U\to \mc{U}_k$.
Given an analytic space $\mf{Z}$, a \emph{formal affinoid covering} of $\mf{Z}$ is an admissible affinoid covering $\mf{U}=\{U_i\}_{i\in I}$ such that $U_i \cap U_j$ is a finite union of subdomains of $\mf{Z}$ of the form $\red_{\mc{U}}^{-1}(\widetilde{V})$ for a Zariski open $\widetilde{V}$ in the canonical reduction $\widetilde{U_i}$ of $U_i$.
With every formal affinoid covering of $\mf{Z}$ one can associate a canonical formal model $\mc{Z}$ of $\mf{Z}$.
The special fiber $\mc{Z}_k$ of $\mc{Z}$ is obtained by pasting the canonical reductions $\widetilde{U_i}$ along the open subsets whose inverse image under the reduction map is $U_i \cap U_j$.
This procedure gives rise to a canonical reduction map $\red_\mc{Z}: \mf{Z} \to \mc{Z}_k$.
\begin{example}
Let $X$ be a $K$-curve, $X^{\mathrm{an}}$ its analytification, and $\mc{X}$ a semi-stable model of $X$.
Then, any formal affine covering of $\mc{X}$ induces on $X^{\mathrm{an}}$ a formal affinoid covering.
The resulting canonical reduction map coincides with the map $\red_\mc{X}$ we refer to in \S\ref{Smodels}.
\end{example}

Let $(\mf{Z}, \mf{U})$ be a pair consisting of an analytic space and a formal affinoid covering on it.
The \emph{inverse image topology} on $\mf{Z}$ induced by the model $\mc{Z}$ arising from $\mf{U}$ is the one whose open sets are inverse images under the map $\red_\mc{Z}$ of Zariski-open subsets of $\mc{Z}_k$.
Let us consider the analytic space $\mf{D}$ and its formal affinoid
covering $\{U_e\}_{e\in E}$ where $E$ is the set of edges of the tree
$\Sigma_{\mf{D}}$ and $U_e$ is the affinoid subdomain of $\mf{D}$
consisting of all the points in $\mf{D}$ that retract to $e$ under the map $\rho_X$ introduced in Definition \ref{def:skeleton}.
The formal model of $\mf{D}$ associated with this covering coincides with the formal scheme $\mc{D}$ defined by Mumford.

\begin{remark}\label{Rinfinitemodels}
Let $\mf{X}$ be a smooth $K$-analytic curve.
From an infinite subset of type 2 points in $\mf{X}$ it is not always possible to construct a formal model of $\mf{X}$ and therefore there is not a straightforward generalization of Proposition \ref{prop:models} in this context.\footnote{Let us mention that such a generalization does nevertheless exist, and was proved by Ducros \cite[Theorem 6.3.15]{Ducros14}.} 
However, given a normal formal model $\mc{X}$ of $X^{\mathrm{an}}$, a
generic point $\eta \in \mc{X}_k$, and the reduction map $\red_\mc{X}:
\mf{X} \to \mc{X}_k$ constructed above, one has that
$\red_{\mc{X}_k}^{-1}(\eta)$ consists of a unique type 2 point
(which by Remark~\ref{Rgeometric} is the same thing as a geometric valuation on $K(X)$).
Hence, to any formal model one can associate a possibly infinite set
of type 2 points in $\mf{X}$ (equivalently, a possibly infinite set of geometric
valuations on $K(X)$).
\end{remark}

Let us now generalize the construction of the model $\mc{D}$.
Let $\mc{X}'$ be a semi-stable model of $X$, and consider the formal
affinoid covering $\mf{V}=\{V_i\}_{i\in I}$ on $X^{\mathrm{an}}$
containing all the affinoid domains of the form $V_i:=\red_{\mc{X}'}^{-1}(\widetilde{V_i})$ for $\widetilde{V_i}$ running over all connected Zariski-open subsets of $\mc{X}'_k$.
Then, we can build a formal affinoid covering of $\mf{D}$ as follows.
If $V_i \in \mf{V}$ is a contractible affinoid domain, then $p^{-1}(V_i)$ is a disjoint union $\coprod_{j\in J_i} U_{ij}$ of affinoid domains of $\mf{D}$ such that $U_{ij} \cong V_j$ for every $j\in J_i$.
If $V_i$ is not contractible, then it follows from the semi-stability of $\mc{X}'$ that it can be decomposed as a union of contractible affinoids and affinoids that retract on a loop.
We can then suppose without loss of generality that the affinoid $V_i$ admits a continuous retraction on a loop or, in other words, that its canonical reduction is an irreducible affine curve with a unique nodal singularity.
In this case, the space $p^{-1}(V_i)$ is connected, but not affinoid.
However, $p^{-1}(V_i)$ is a union $\bigcup_{j\in J_i} U_{ij}$ of affinoids satisfying the following properties:

\begin{itemize}
\item The projection $p(U_{ij})$ is equal to $V_i$;
\item The boundary of $U_{ij}$ in $\mf{D}$ consists of two distinct points.
\end{itemize}

The collection $\{U_{ij}\}_{i\in I, j \in J_i}$ is a formal affinoid covering of $\mf{D}$, and then gives rise via Bosch--L\"{u}tkebohmert's theory to a formal model of $\mf{D}$, that we denote by $\mc{D}'$.
Thanks to Remark \ref{Rinfinitemodels}, one can associate a set of type 2 points $\mc{Q}_\mc{D'} \subset \mf{D}$ with the formal model $\mc{D}'$.
If we denote by $\mc{Q}_\mc{X'}$ the set of type 2 points of $X^{\mathrm{an}}$ associated with $\mc{X}'$, then we have that $p^{-1}(\mc{Q}_\mc{X'}) = \mc{Q}_\mc{D'}$.
This follows from the definition once we remark that $\mc{Q}_\mc{X'}$ (respectively $\mc{Q}_\mc{D'}$) consists precisely of the boundary points of the affinoid domains in the covering defining the formal model $\mc{X'}$ (respectively $\mc{D'}$).

\begin{lemma}\label{lem:formallift}
Let $\mc{X}'$ be a semi-stable model of a Mumford curve $X$, and let $\mc{D}'$ be the model of $\mf{D}$ constructed from $\mc{X}'$ as above.
Then, the uniformization map $p:\mf{D} \to X^{\mathrm{an}}$ of Theorem
\ref{thm:unif} is continuous for the inverse image topologies induced
by $\mc{D}'$ and $\hat{\mc{X}}'$ respectively, and extends to a local isomorphism of formal schemes $\varpi \colon \mc{D}' \to \hat{\mc{X}'}$.
\end{lemma}
\begin{proof}
Let $V \subset X^{\mathrm{an}}$ be of the form $\red_{\mc{X}'}^{-1}(\widetilde{V})$ for some Zariski open subset $\widetilde{V}$ of $\mc{X}'_k$.
Then, by decomposing $\widetilde{V}$ into its connected components, $V$ can be written as a union of elements $V_i$ of the formal affinoid covering $\mf{V}$ associated with the formal model $\widehat{\mc{X}}'$.
By construction of the formal model $\mc{D}'$, the preimage $p^{-1}(V)=~\bigcup p^{-1}(V_i)$ is an infinite union of affinoid domains of the form $\red_{\mc{D}'}^{-1}(\widetilde{U})$ for some connected Zariski open subset $\widetilde{U}$ of $\mc{D}'_k$, hence $p$ is continuous for the inverse image topology.

Let us show that this continuity is enough to extend $p$ to a morphism of formal models $\varpi\colon \mc{D}' \to \hat{\mc{X}'}$.
First of all, the functoriality of the reduction ensures the existence of a morphism of special fibers $p_k:\mc{D}'_k \to \mc{X}'_k$ compatible with the reduction maps as illustrated in the following commutative diagram:
\[\xymatrix{ \mf{D} \ar[r]^p\ar@{->}[d]_{\red_{\mc{D}'}} & X^{\mathrm{an}} \ar@{->}[d]^{\red_{\mc{X}' }} \\
           \mc{D}'_k \ar[r]^{p_k}  &  \mc{X}'_k. }\]
An affine Zariski open $\Spec(A)$ of $\mc{X}'_k$ can always be completed to get an affine formal subscheme $\Spf(\mc{A}) \subset \mc{X}'$.
The Zariski open $p_k^{-1}(\Spec(A))$ can be covered by affine opens $\Spec(B_i) \subset \mc{D}'_k$, which lift to formal affine subschemes $\Spf(\mc{B}_i) \subset \mc{D}'$.
By glueing all the induced morphisms $\Spf(\mc{B}_i) \to \Spf(\mc{A})$, we can construct a morphism $\varpi\colon \mc{D}' \to \hat{\mc{X}'}$ that lifts $p_k$.
The continuity for the inverse image topology then ensures that the generic fiber of $\varpi$ is the uniformization map $p:\mf{D} \to X^{\mathrm{an}}$, as required.

Since $p$ is a local isomorphism of analytic spaces, after possibly
replacing $\mf{V}$ with a union of smaller affinoid subsets that are open for the inverse image topology we can suppose that $p_{|\mf{V}}$ is an isomorphism into its image.
Hence, $\varpi$ is a local isomorphism.
\end{proof}

\begin{remark}
The construction of the formal scheme $\mc{D}'$ can be associated with any semi-stable model $\mc{X}'$ and is of particular interest in some special cases.
When the genus of $X$ is at least 2 and $\mc{X}$ is the stable model of $X$, one gets the formal model $\mc{D}$ studied by Mumford (see Remark \ref{rmk:Mumformal}). 
When $\mc{X}'$ is the model considered in Proposition \ref{prop:mult1}, that is, the minimal desingularization of $\mc{X}$, one gets that $\mc{D}'$ is the minimal desingularization of $\mc{D}$.
In fact, the minimal desingularization of a semi-stable model is obtained by blowing up $\mc{X}$ a finite number of times at its singular double points $P_1, \dots, P_n$. As a result, if $\mc{Q}$ is the set of type 2 points of $X^{\mathrm{an}}$ associated with $\mc{X}$, then the set $\mc{Q}'$ of type 2 points of $X^{\mathrm{an}}$ associated with $\mc{X}'$ is obtained from $\mc{Q}$ by adding a finite number of type 2 points lying on the formal fibers $\red^{-1}(P_i)$.
Since $\mc{D}'$ induces by construction the set $p^{-1}(\mc{Q}')$ of type 2 points of $\mf{D}$, it is a regular model.
Removing any subset of type 2 points from $p^{-1}(\mc{Q}')$ would create a singularity, hence $\mc{D}'$ is the minimal desingularization of $\mc{D}$.
Thanks to Lemma \ref{lem:formallift}, in all these instances we can extend $p:\mf{D} \to X^{\mathrm{an}}$ to a local isomorphism of formal schemes $\varpi':\mc{D}'\to \mc{X}'$.
\end{remark}

\begin{prop}\label{prop:semistable}
Let $X$ be a Mumford curve, $\mc{X}'$ a semi-stable model of $X$, and
$\mc{D}'$ the corresponding formal model of $\mf{D}$. Then, for every closed point $x \in \mc{D}'$, the map $\varpi':\mc{D}' \to \mc{X}'$ that lifts $p$ induces an isomorphism of local rings
\[ \widehat{\calO}_{\mc{D}',x} \cong \widehat{\calO}_{\mc{X}',\varpi'(x)}. \]

\end{prop}
\begin{proof}
Let $\mf{U}$ be the formal affinoid covering of $\mf{D}$ associated with the formal model $\mc{D}'$.
Let $U$ be an affinoid in this covering containing $\red_{\mc{D}'}^{-1}(x)$. Since $\mc{X}'$ is semi-stable, then it dominates the minimal semi-stable model, so that $\mf{U}$ is a refinement of the smallest affinoid covering of $\mf{D}$ made of discs and annuli.
In particular, $U$ is contained in a fundamental domain for the action of the Schottky group, hence $p(U) \cong U$.
By construction, we have that the structure sheaf of $\mc{D}'$ is obtained as the sheaf $\calO^\circ$ of functions bounded by 1, and the lift $\varpi'$ is compatible with this operation, hence it induces an isomorphism \[\varpi'(\Spf(\calO^\circ(U))) \cong \Spf(\calO^\circ(U)),\] which in turn induces an isomorphism of the completed local rings $\widehat{\calO}_{\mc{D}',x} \cong \widehat{\calO}_{\mc{X}',\varpi'(x)}$.
\end{proof}

\begin{lemma}\label{Lsamelocalring}
Let $\mf{E}$ be a subset of $\proj^{1, \mathrm{an}}_K$ consisting of type 1 points and let $\mc{Y}$ be a formal scheme whose generic fiber is $\mc{Y}_K \cong \proj^{1, \mathrm{an}}_K \setminus \mf{E}$.
Let $\Sigma$ be a finite set of irreducible components of $\mc{Y}$,
let $S$ be the corresponding set of type 2 points of $\mc{Y}_K$, and
let $\mc{Y}_S$ be the normal model of $\proj^1_K$ corresponding to $S$
via Proposition~\ref{prop:models}.
Let $\mc{Y}_\Sigma$ be the set of closed points of $\mc{Y}$ that lie only on components in $\Sigma$ and let $\phi$ be the injection from $\mc{Y}_\Sigma$ to the set of closed points of $\mc{Y}_S$ induced by the inclusion $\mc{Y}_K \subset \proj^{1, \mathrm{an}}_K$.
Then for every point $x \in \mc{Y}_\Sigma$ there is an isomorphism of local rings 
\[ \widehat{\calO}_{\mc{Y},x} \cong \widehat{\calO}_{\mc{Y}_S,\phi(x)}. \] 
\end{lemma}

\begin{proof}

Let us fix a point $x \in \mc{Y}_\Sigma$. By definition, its formal fiber $V_x=\red_\mc{Y}^{-1}(x)$ is a subset of $\proj^{1, \mathrm{an}}_K \setminus \mf{E}$. As a result, the inclusion $\mc{Y}_K \subset \proj^{1, \mathrm{an}}_K$ induces an equality between the ring of analytic functions on $V_x$ and the ring of analytic functions on the formal fiber $V_{\varphi(x)}:=\red_{\mc{Y}_S}^{-1}(\varphi(x))$.
The local ring $\widehat{\mc{O}}_{\mc{Y},x}$ is a reduced special $R$-algebra whose generic fiber is the semi-affinoid analytic space $V_x$.
We can then apply \cite[Theorem 2.1]{Martin17} to show that $\widehat{\mc{O}}_{\mc{Y},x}$ is isomorphic to the subring of bounded functions of $\mc{O}_{V_x}(V_x)$.
The same holds for the local ring $\widehat{\calO}_{\mc{Y_S},\varphi(x)}$,
\end{proof}

\subsubsection{Lifting automorphisms of curves}\label{Slifting}
Let $X$ be a $K$-curve, and let $X^{\mathrm{an}}$ be its Berkovich analytification.
The analytic curve $X^{\mathrm{an}}$ is a path connected and locally contractible topological space by \cite[Corollary 4.3.3]{Berkovich90}.
As such, it has a universal cover $p\colon \frakY \to X^{\mathrm{an}}$, uniquely determined up to deck transformations of the universal covering space $\frakY$.
The cover $p$ is a local homeomorphism, hence we can identify for every $y\in \frakY$ a sufficiently small open neighborhood $V$ of $y$ with its image $p(V)$ in $X^{\mathrm{an}}$ and define on $\frakY$ the unique analytic structure that makes the universal cover $p$ into a local isomorphism of locally ringed spaces. 
The analytic curves $\frakY$ and $X^{\mathrm{an}}$ are isomorphic if and only if $X^{\mathrm{an}}$ is simply connected. 
Note that, in general, $\frakY$ will not be compact, and therefore not the analytification of a projective $K$-curve.

We now concern ourselves with lifting automorphisms from a Berkovich
curve to its universal cover.
We start by remarking that for $K$-curves we have a canonical inclusion $\Aut(X) \subset \Aut(X^{\mathrm{an}})$ that will allow us to identify the automorphisms of a curve with those of its analytification.
If $\delta \in \Aut(X)$ is a $K$-linear automorphism, we call a \emph{lift to the universal cover} any automorphism $\widetilde{\delta} \in \Aut(\frakY)$ that makes the following diagram commute
\[\xymatrix{ \frakY \ar[r]^p\ar[d]_{\widetilde{\delta}} & X^{\mathrm{an}} \ar[d]^{\delta} \\
           \frakY \ar[r]^p  &  X^{\mathrm{an}}. }\]

\begin{prop}\label{prop:generallifting}
Let $X$ be a smooth connected $K$-curve, and let $\delta\in \Aut(X)$ be a $K$-linear automorphism of $X$.
Then, there exists a lift $\widetilde{\delta} \in \Aut(\frakY)$ of $\delta$ to the universal cover $\frakY$ of $X^{\mathrm{an}}$.
\end{prop}
\begin{proof}
Let us fix $x\in X^{\mathrm{an}}$ and choose $y,y' \in \mf{Y}$ such that $p(y)=x$ and $p(y')=\delta(x)$.
Then, the map $p:(\mf{Y},y') \to (X^{\mathrm{an}},\delta(x))$ is a covering and $\delta \circ p:(\mf{Y},y) \to (X^{\mathrm{an}},\delta(x))$ is a local homeomorphism of pointed topological spaces.
The space $\mf{Y}$ is connected, and therefore by \cite[Theorem 3.2.1]{Berkovich90} it is also path-connected and locally path-connected.
The continuous map $\delta \circ p$ can then be lifted uniquely to a continuous map $\widetilde{\delta}:(\mf{Y},y') \to (\mf{Y},y)$.
To show that this is a bijection, observe that we can repeat the same argument above and lift $\delta^{-1} \circ p:(\mf{Y},y) \to (X^{\mathrm{an}},x)$ to a continuous map $(\mf{Y},y) \to (\mf{Y},y')$ which is clearly the inverse of $\widetilde{\delta}$.
Finally, we can upgrade the homeomorphism $\widetilde{\delta}$ to an automorphism of analytic spaces. Since both $\delta \circ p$ and $p$ are local isomorphisms of analytic spaces, every point $z \in \mf{Y}$ has a neighborhood $Y_z$ such that $\widetilde{\delta}$ is an isomorphism when restricted to $Y_z$.
Hence $\widetilde{\delta}$ is a homeomorphism that is also a local isomorphism, and therefore it is an automorphism of the analytic space $\mf{Y}$.

\end{proof}

Let $L/K$ be a finite Galois extension and let $X$ be a Mumford curve over $L$.
By Schottky uniformization (Theorem \ref{thm:unif}), the universal
cover $\frakY$ is isomorphic to the space $\P^{1,\mathrm{an}}_L
\setminus \mathcal{L}$, where $\mathcal{L}$ is the limit set of the
Schottky group associated with $X$.  Note that we can consider any $L$-analytic space as a $K$-analytic
space by post-composing with the canonical map $\calM(L) \to
\calM(K)$.
For the scope of this paper, we are interested in relaxing the $L$-linearity requirement for an automorphism of $L$-analytic spaces to a weaker condition that reflects compatibility with the structure of the Galois group $\Gal(L/K)$, as made precise in the following definition.
\begin{defn}
Let $\frakX$ be a $L$-analytic space, and let $\sigma \in \Gal(L/K)$. 
A $K$-linear automorphism $\delta$ of $\frakX$ is said to be $\sigma$-semilinear if it fits in a commutative diagram of the form:
\[\xymatrix{ \frakX \ar[d]_{\delta} \ar[r] & \calM(L) \ar[r]\ar[d]_{\sigma} & \calM(K)\ar[d]_{\mathrm{id}}  \\
           \frakX \ar[r] & \calM(L) \ar[r] & \calM(K). }\]
Here and throughout, by abuse of notation, we use $\sigma: \calM(L) \to \calM(L)$ to
mean the map $(\sigma^{-1})^*$.  This has the consequence that if
$\frakX$ is in fact defined over $K$, then the 
$\sigma$-action on $\frakX$ that acts by $\sigma$ on the coordinates of $L$-points is $\sigma$-semilinear.

A $K$-linear automorphism of $\frakX$ is said to be $L/K$-semilinear if it is $\sigma$-semilinear for some $\sigma \in \Gal(L/K)$.
In particular, every $L$-linear automorphism is $L/K$-semilinear.
\end{defn}

Let $\delta$ be a $\sigma$-semilinear automorphism of
$X^{\mathrm{an}}$.
Since by definition $\delta$ is $K$-linear, we can apply Proposition \ref{prop:generallifting} to lift it to a $K$-linear automorphism $\widetilde{\delta}$ of the universal covering space $\frakY$, in such a way that the following diagram commutes:

\begin{equation}\label{diagram:unicover}
\xymatrix{ \frakY \ar[r]^p\ar[d]_{\widetilde{\delta}} & X^{\mathrm{an}} \ar[d]_{\delta} \ar[r] & \calM(L) \ar[r]\ar[d]_{\sigma} & \calM(K)\ar[d]_{\mathrm{id}}  \\
           \frakY \ar[r]^p  &  X^{\mathrm{an}} \ar[r] & \calM(L) \ar[r] & \calM(K).}
\end{equation}

From this, we automatically obtain that the automorphism $\widetilde{\delta}$ is $L/K$-semilinear.


\begin{prop}\label{prop:extensionofliftings}
Every $\sigma$-semilinear automorphism of $\frakY$ 
can be extended to a $\sigma$-semilinear automorphism of $\P^{1,\mathrm{an}}_L$.
\end{prop}

\begin{proof}
Let $\tilde{\delta}$ be a $\sigma$-semilinear automorphism of $\frakY$.
Let $\tilde{\tilde{\delta}} = \tilde{\sigma}^{-1}|_{\mf{Y}} \circ \tilde{\delta}$, where $\tilde{\sigma}$ is the purely arithmetic action of $\sigma$ on
$\proj^{1, \mathrm{an}}_L$.
Then $\tilde{\tilde{\delta}} \colon \mf{Y} \to \proj^{1, \mathrm{an}}_L$ is
an $L$-linear open immersion. 
The limit set $\mc{L}$ is compact, because it is the complement of the domain of discontinuity for the action of a Schottky group, which is open.
This means that we can apply the same argument of the proof of \cite[Lemma
6.5.1]{PoineauTurchetti21}\footnote{As stated, this lemma applies
  only to an \emph{automorphism} of $\mf{Y}$, but the proof goes
  through verbatim for an open immersion $\mf{Y} \to \proj^{1,
    \mathrm{an}}_L$. In fact, the proof in \cite{PoineauTurchetti21} makes an erroneous claim
 that $\tau^{-1} \circ \sigma$ is an automorphism
  of $O$ (which plays the role of $\mf{Y}$), when in fact it is only an open immersion from $O$ to
  $\proj^1$, but the open immersion is good enough for the rest of the
  proof to work.} to conclude that $\tilde{\tilde{\delta}}$ is induced by an element of $PGL_2(L)$.  By
abuse of notation, we denote the corresponding $L$-automorphism of
$\proj^{1, \mathrm{an}}_L$ by $\tilde{\tilde{\delta}}$ as well. 
The $\sigma$-semilinear
automorphism of $\proj^{1, \mathrm{an}}_L$ that agrees with
$\tilde{\delta}$ on $\mf{Y}$ is hence given by $\tilde{\sigma} \circ \tilde{\tilde{\delta}}$.
%
\end{proof}


\subsection{Mac Lane valuations}\label{Smaclane}
\subsubsection{Valuation theory}\label{Svaluation}
Recall from Remark~\ref{Rgeometric} that a geometric valuation on
$K(x)$ is a discrete valuation that restricts
to the valuation $v_K$ on $K$ and whose residue field is finitely
generated over $k$ with transcendence degree $1$. 
There is a partial
order on geometric valuations given by $v \preceq w$ if and only if
$v(f) \leq w(f)$ for all $f \in K[x]$.

Let $v_0$ be the
\emph{Gauss valuation} on $K(x)$, uniquely determined by $v_0(\sum_i a_ix^i) =
\inf_i(v_K(a_i))$ on polynomials.  By \cite[Corollary~7.4]{FGMN} and \cite[Theorem 8.1]{MacLane}, or
\cite[Theorem 4.31]{Ruth}, every geometric valuation $v$ with $v
\succeq v_0$ can be written as a so-called \emph{Mac Lane valuation}
in the form
\begin{equation}\label{Emaclane}
v = [v_0,\, v_1(\phi_1) = \lambda_1,\, \ldots,\, v_n(\phi_n) =
\lambda_n]
\end{equation}
where $n \geq 0$, the $\phi_i$ are monic polynomials in $K[x]$ of
increasing degree and the $\lambda_i$ are positive rational numbers.  The meaning of the
notation is this: to calculate $v(f)$, write out the $\phi_n$-adic
expansion $$f = a_e\phi_n^e + a_{e-1} \phi_n^{e-1} + \cdots + a_0.$$
Then recursively calculate $v(f) = \inf_{0 \leq i \leq e} v_{n-1}(a_i)
+ i \lambda_n$, where $v_{n-1}$ is shorthand for the valuation
$$[v_0,\, v_1(\phi_1) = \lambda_1,\, \ldots,\, v_{n-1}(\phi_{n-1})
= \lambda_{n-1}].$$  In general, if $v$ is given as in
(\ref{Emaclane}) and $0 \leq i \leq n$, we write $v_i$ for the valuation $[v_0,\, \ldots,\,
v_i(\phi_i) = \lambda_i]$.

Not all choices of the $\phi_j$ and $\lambda_j$ give
rise to a geometric valuation (or a valuation at all), but if $v$ as
in (\ref{Emaclane}) is a geometric valuation, then increasing
$\lambda_n$ results in another geometric valuation.  For
fuller background, see \cite{MacLane} or \cite{Ruth}.  When we use the
term \emph{Mac Lane valuation}, it is assumed that we are talking about a
geometric valuation. 

\subsubsection{Relationship with normal models}\label{Snormal}
Let $\mc{Y}$ be a normal model of $\proj^1_K$, and choose a rational
function $x$ on $\proj^1_K$ generating the function field.  Recall from
Proposition~\ref{prop:models} that there is a corresponding set $S$ of
geometric valuations on $K(x)$.  We say that the valuations $v \in S$ are
\emph{included} in $\mc{Y}$.  If $S = \{v\}$, we call the corresponding model the
\emph{$v$-model} of $\proj^1_K$.  If $v$ and $w$ are included in
$\mc{Y}$, we say that $v$ is \emph{adjacent} to $w$ if $v \prec w$ and
there is no $z$ included in $\mc{Y}$ with $v \prec z \prec w$, or if $w \prec v$ and
there is no $z$ included in $\mc{Y}$ with $w \prec z \prec v$.  

One can always make a linear
fractional change of variable in $x$ such that every $v \in S$
satisfies $v(x) \geq 0$ (this is tantamount to making sure the
function $x$ does not have a pole at any generic point of the special
fiber).  This is equivalent to $v \succeq v_0$.  In particular, we may assume that $S$ consists of
Mac Lane valuations.

\begin{lemma}\label{Lcalculatemultiplicity}
Let $\ol{V}$ be an irreducible component of the special fiber of $\mc{Y}$ corresponding to a
Mac Lane valuation $v = [v_0,\, \ldots,\, v_n(\phi_n) = \lambda_n]$, and
let $m$ be the lcm of the  denominators of all the $\lambda_i$, when
written in lowest terms.
Then the multiplicity $m_{\ol{V}}$ of $\ol{V}$ in the special fiber
equals $m$.
\end{lemma}

\begin{proof}
The value group of $v$ is $\frac{1}{m}\ints$, so the lemma follows from 
(\ref{Evaldef}).
\end{proof}

Given two Mac Lane valuations $w$ and $w'$, we write $\inf(w, w')$ for
the maximal (for $\preceq$) Mac Lane valuation $v$ satisfying
$v \preceq w$ and $v \preceq w'$.  By \cite[Proposition 2.24]{KW},
$\inf(w, w')$ exists and is unique.

\begin{lemma}\label{Linfclosed}
  Every regular model $\mc{Y}$ of $\proj^1_K$ is \emph{inf-closed}.  That is,
  if $w$ and $w'$ are two valuations included in $\mc{Y}$, then
  $\inf \{w, w'\}$ is also included in $\mc{Y}$.
\end{lemma}

\begin{proof}
Let $v = \inf \{w, w'\}$.  Write $$v = [v_0,\, v_1(\phi_1) =
\lambda_1,\, \ldots,\, v_n(\phi_n) = \lambda_n].$$  We may assume $v
\neq w$ and $v \neq w'$.  It is not possible that both $w(\phi_n) >
\lambda_n$ and $w'(\phi_n) > \lambda_n$ because then $v$ would not
equal $\inf(w, w')$, as $\lambda_n$ could be increased.
So by \cite[Proposition~4.35]{Ruth}, at least one of $w$ or $w'$, say
$w$, is of the form
$$w = [v_0,\, w_1(\phi_1) = \lambda_1,\, \ldots, w_n(\phi_n) =
\lambda_n,\, \ldots,\, w_m(\phi_m) = \lambda_m]$$ with $m > n$.  By \cite[Theorem
7.8]{ObusWewers}, $v$ is included in the minimal regular resolution of the
$w$-model of $\proj^1$.  Since $\mc{Y}$ is a regular resolution of the
$w$-model of $\proj^1_K$, we thus have that $v$ is included in $\mc{Y}$. 
\end{proof}

\begin{lemma}\label{Lonlyvi}
Let $\mc{X}$ be a normal model of $\proj^1_K$ corresponding to a
set ${v^1, v^2, \ldots, v^r}$ of Mac Lane valuations, where
$$v^j = [v_0^j, v_1^j(\phi^j_1) = \lambda^j_1, \ldots,
v_{n_j}^j(\phi^j_{n_j}) = \lambda^j_{n_j}]$$ and $v_0^j = v_0$ for all $j$.
If $\mc{X}' \to \mc{X}$ is the minimal regular resolution, then every
principal component of the special fiber of $\mc{X}'$ corresponds to
some $v_i^j$, for $j \in \{1, \ldots, r\}$ and $i \in \{0, \ldots, n_j\}$. 
\end{lemma}

\begin{proof}
Let $\mc{X}_j$ be the $v^j$-model of $\proj^1_K$, and let $\mc{X}_j'
\to \mc{X}_j$ be its minimal regular resolution.  
  By \cite[Lemma~5.3]{OS1}, the set $S$ of Mac Lane valuations
  corresponding to $\mc{X}'$ is $\bigcup_{j=1}^r S_j$, where $S_j$ is
  the set of Mac Lane
  valuations corresponding to the $\mc{X}_j'$.

Suppose $v \in S \setminus \bigcup_{i, j} \{v_i^j\}$ 
corresponds to a component $\ol{V}$ of the special fiber of $\mc{X}'$.
By \cite[Lemma~7.1]{ObusWewers}, $\ol{V}$ has genus zero.  The set $S$
satisfies \cite[Assumption~3.4(a)]{KW} automatically, and by
Lemma~\ref{Linfclosed}, it also satisfies \cite[Assumption~3.4(b)]{KW}.  By
\cite[Proposition~3.5]{KW}, the only irreducible components
intersecting $\ol{V}$ correspond to valuations $w$ adjacent to $v$ in $S$.  By
\cite[Proposition~2.25]{KW} there is only one such $w$ with $w \prec
v$.  We claim there cannot exist $w_1 \neq w_2$ adjacent to $v$ in $S$ with $v \prec w_1$ and
$v \prec w_2$, and thus $\ol{V}$ is not prinicipal.

To prove the claim, write $v = [v_0,\, v_1(\phi_1) = \lambda_1,
\ldots,\, v_n(\phi_n) = \lambda_n]$, and assume that $w \succ v$ with
$w$ adjacent to $v$ in $S$.  By
\cite[Proposition~4.35]{Ruth}, $w$
is of the form
$$[v_0,\, v_1(\phi_1) = \lambda_1,\, \ldots,\, w_n(\phi_n) =
\mu_n,\, v_{n+1}(\phi_{n+1}) = \lambda_{n+1},\, \ldots,\,
w_m(\phi_m) = \lambda_m]$$
with $\mu_n \geq \lambda_n$ and $m \geq n$.  Since $w \in S$, it is in
$S_j$ for some $j$. By
\cite[Theorem~7.8]{ObusWewers}, $w_n \in
S_j$, and thus $w_n \in S$.  Since $w$ is adjacent to
$v$ and $v \preceq w_n \preceq w$, either $w = w_n$ or $v  = w_n$ (with
$m > n$).  If
$v = w_n$ with $m > n$, then \cite[Theorem~7.8]{ObusWewers} implies
that $v = w_n$ is in fact $v_n^j$, a contradiction. So $w = w_n$, and
the fact that $w$ is adjacent to $v$ in $S$ uniquely determines
$\mu_n$, and thus $w$.  This proves the claim.
\end{proof}

\begin{prop}\label{Pmaclaneresolution}
  Let $\mc{X}$ be a normal model of $\proj^1_K$, and let $\mc{X}'$ be its
  minimal regular resolution.  If $m$ is the
  lcm of the multiplicities of the components of the special fiber of $\mc{X}$, then the stability
  index of $\mc{X}'$ divides $m$.
\end{prop}

\begin{proof}
Let $S$ be the set of Mac Lane valuations corresponding to the model
$\mc{X}$.   If $w$ is a Mac Lane valuation
corresponding to a principal component of $\mc{X}'$, then
Lemma~\ref{Lonlyvi} implies that there exists a valuation $v = [v_0,\,
v_1(\phi_1) = \lambda_1,\, \ldots,\, v_n(\phi_n) = \lambda_n] \in S$
such that $w = v_i$ for some $i \leq n$.
Now, the multiplicity $m_v$ of the component corresponding to $v$ is the lcm of the denominators of the
$\lambda_j$, $j \leq n$.  Since the multiplicity $m_w$ of the component
corresponding to $w$ is the lcm of the denominators of the
$\lambda_j$ for $j \leq i$, we have $m_w \mid m_v$ as desired.
\end{proof}

\begin{corollary}\label{Carithmeticaction}
Let $L/K$ be a finite Galois extension, write $\proj^1_L =
\Proj L[x_0, x_1]$, and let $x = x_1/x_0$.   Suppose $\Gal(L/K)$ acts
on $\proj^1_L$ in the standard manner via its action on $L$, leaving
$x$ constant.  Let $\mc{Y}$ be a semi-stable model
of $\proj^1_L$ over $\mc{O}_L$ such that the action of $\Gal(L/K)$ extends to
$\mc{Y}$.  Then the stability index of the minimal regular resolution
$\mc{X}'$ of $\mc{X} \colonequals
\mc{Y}/\Gal(L/K)$ divides $[L:K]$.
\end{corollary}

\begin{proof}
The scheme $\mc{X}$ is a model of $\proj^1_K$ such that all irreducible
components of its special
fiber have multiplicity dividing $[L:K]$.  The corollary now follows
from Proposition~\ref{Pmaclaneresolution}.

\end{proof}

\section{Descent and fields of definition}\label{Sdescent}
In \S\ref{Smumforddescent}, we give a criterion for a Mumford curve
over a Galois extension $L/K$ to descend to a curve over $K$.  In
\S\ref{Sfieldofdefinition}, we discuss what it means for a type 2
point of the analytification of a Mumford curve $X$ over $L$ to be defined over a subfield of $L$,
and we give some consequences for the $\Gal(L/K)$-action on certain semistable
models of $X$.

We take the $\Gal(L/K)$-action on $\proj^{1, \mathrm{an}}_L$ to be the
usual action on points.  

\subsection{Descent of Mumford curves}\label{Smumforddescent}
The following proposition shows that a Mumford curve over $L$ descends to $K$ if its Schottky group is Galois invariant, after possibly performing a change of parameter.
\begin{prop}\label{Pdescentcriterion}
Let $L/K$ be a Galois extension.
Let $Y$ be a Mumford curve over $L$.  Denote by $\Gamma \subset \PGL_2(L)$ the corresponding Schottky group and by $\calL \subset \P^{1, \mathrm{an}}_L(L)$ the corresponding limit set, in such a way that $Y$ is the quotient of $\P^{1, \mathrm{an}}_L \setminus \calL$ by the action of $\Gamma$.
Under the natural actions of $\Gamma$ and $\Gal(L/K)$ on $\P^{1,
  \mathrm{an}}_L$, we consider the following three statements:
\begin{enumerate}
\item There exists a $K$-curve $X$ such that $X_L \cong Y$;
\item After possibly conjugating $\Gamma$ by an element of
  $PGL_2(L)$, for every $\sigma \in \Gal(L/K)$ one has
  $\sigma^{-1}\Gamma\sigma = \Gamma$ (the composition is viewed inside
  the group of $L/K$-semilinear automorphisms of $\proj^1_L$);
\item After possibly conjugating $\Gamma$ by an element of $PGL_2(L)$,
  for every $\sigma \in \Gal(L/K)$ and $\gamma =\begin{bmatrix} a_1 & a_2 \\ a_3 & a_4 \end{bmatrix} \in \Gamma$ one has $\begin{bmatrix} \sigma(a_1) & \sigma(a_2) \\ \sigma(a_3) & \sigma(a_4)\end{bmatrix} \in \Gamma$.
\end{enumerate}
Then $(b) \iff (c) \implies (a)$, and if $Y$ has genus 1, then all
three are equivalent.
Moreover, if $(b)$ holds, then every $\sigma \in \Gal(L/K)$ satisfies $\sigma(\calL)=\calL$.

\end{prop}
\begin{proof}
$(b) \implies (a)$: First we show that condition $(b)$ implies that, up to reparametrization, the limit set $\calL$ is globally fixed by the elements of the Galois group $\Gal(L/K)$, proving the last sentence in the statement of the proposition.
More precisely, let us assume $(b)$ and fix a parameter of $\panL$ that makes the Schottky group $\Gamma$ invariant by $\Gal(L/K)$-conjugation.
 If we write a point $\ell$ of $\calL$ as $\ell = \lim_{i\to \infty}
\gamma_i(x)$ for a certain $x\in \P^{1, \mathrm{an}}_L$, then we
deduce from the condition $\sigma\gamma\sigma^{-1}\in \Gamma$ for all $\gamma \in \Gamma$ that  $\sigma(\ell)=\lim_{i\to \infty} \gamma'_i(\sigma(x))$ for some sequence $(\gamma'_i)_{i\in \N}$ of elements of $\Gamma$, which means that $\sigma(\ell)$ is also in the limit set.
By applying the same argument to $\sigma^{-1}$ one gets that $\sigma(\calL)=\calL$. This ensures that the arithmetic action of $\Gal(L/K)$ on $\panL$ restricts to an action on the universal cover $\mf{D}$ of $Y$: 
\[\psi\colon \Gal(L/K) \times \mf{D} \to \mf{D}.\]
Let us denote by $\psi_\sigma$ the automorphism of $\mf{D}$ defined by $\psi_\sigma(x) = \psi(\sigma, x)$.
Since $\psi$ comes from the natural arithmetic action, we have that $\psi_\sigma$ is obtained as the pullback of $\sigma$ in the sense that it fits in the following Cartesian diagram:
\[\xymatrix{ \ar @{} [dr] |{\square} \mf{D}^\sigma \ar[d] \ar[r]^{\psi_\sigma} & \mf{D}
\ar[d] \\
           \calM(L) \ar[r]^{\sigma} & \calM(L).}\] 
Moreover, using the assumption $(b)$ we have that $\psi_\sigma(\gamma(x))=\gamma'(\psi_\sigma(x))$ for some $\gamma' \in \Gamma$, which ensures that $\psi_\sigma$ induces a $\sigma$-semilinear automorphism $\varphi_\sigma$ of the Mumford curve $Y$.
From the fact that the uniformization map is $L$-linear, we obtain
that the map $\varphi_\sigma$ can be obtained as the pullback of $\sigma$ in the sense that it fits in the following Cartesian diagram:
\[\xymatrix{ \ar @{} [dr] |{\square} Y^\sigma \ar[d] \ar[r]^{\varphi_\sigma} & Y
\ar[d] \\
           \Spec(L) \ar[r]^{\sigma} & \Spec(L).}\]
           In particular, we have that $Y^\sigma \cong Y$ for all $\sigma \in \Gal(L/K)$.
Let $\tau\in \Gal(L/K)$, and let $\varphi_{\tau}^{\sigma}\colon
Y^{\sigma\tau} \to Y^{\sigma}$ be the Cartesian pullback via $\Spec(L)
\stackrel{\tau}{\to} \Spec(L)$. By compatibility with the arithmetic action, we have that $\varphi_{\sigma\tau}=\varphi_\sigma \circ \varphi_\tau^{\sigma}$.
Using Weil's criterion (\cite[Theorem 1]{Weil56} in the slightly
weaker form provided by \cite[Theorem (1.9)]{Koeck}), we get that the
field of definition of $Y$ is contained in $K$, which corresponds to
proving $(a)$.\\

The equivalence between $(b)$ and $(c)$ is a direct computation left to the reader.\\

$(a) \implies (c)$ in genus 1:
Assume $(a)$ and consider the base change $X_L$,
which is a Tate curve associated to a Schottky group with a single
generator $\gamma_1$. After a change of variable we may
assume $\gamma_1 = \begin{pmatrix}
\beta & 0 \\0& 1
\end{pmatrix}$ with $|\beta| < 1$.  The corresponding limit set is $\{0, \infty\}$.  
The base change of $X$ induces a Galois action on $X_L$.
For every $\sigma \in \Gal(L/K)$, its action on $X_L$ can be lifted to a $\sigma$-semilinear automorphism of $\P^{1, \mathrm{an}}_L \setminus \{0,\infty\}$ by Proposition \ref{prop:generallifting}.
Let us call this extension $\tilde{\tau}$ and denote by $\tilde{\sigma}$ the automorphism of $\P^{1, \mathrm{an}}_L$ induced by the purely arithmetic action of $\sigma$. 
Note that, by virtue of $\tilde{\tau}$ being a lifting of an automorphism of $X_L$,
we have that $\tilde{\tau} \Gamma \tilde{\tau}^{-1} = \Gamma$.

Thus, the composition $\tilde{\sigma}^{-1} \circ \tilde{\tau}$ is an $L$-linear automorphism $\gamma_\sigma \in \PGL_2(L)$ satisfying 
\[\gamma_\sigma \Gamma \gamma_\sigma^{-1} = \tilde{\sigma}^{-1} \tilde{\tau} \Gamma \tilde{\tau}^{-1} \tilde{\sigma} = \tilde{\sigma}^{-1} \Gamma \tilde{\sigma}.\]
Since $\Gamma$ is generated by a unique element $\gamma_1$, this relation translates to 
\[\tilde{\sigma}^{-1} \gamma_1 \tilde{\sigma} = \gamma_\sigma \gamma_1^\epsilon \gamma_\sigma^{-1},\] where $\epsilon=\pm 1$.
This results in $\tilde{\sigma}^{-1} \gamma_1 \tilde{\sigma}$ being conjugate by an element of $PGL_2(L)$ to $\gamma_1$ (if $\epsilon=- 1$ this is true because $\gamma_1$ is conjugate to $\gamma_1^{-1}$ by $\begin{pmatrix}
0 & 1 \\1& 0
\end{pmatrix}$). Since the ratio of the eigenvalues of an element of $PGL_2$ is
conjugation-invariant, $\sigma(\beta) \in \{\beta, \beta^{-1}\}$.  But
$\sigma$ preserves absolute values, so $\sigma(\beta) \neq
\beta^{-1}$, implying that $\sigma(\beta) = \beta$.
By repeating the argument with every $\sigma \in \Gal(L/K)$ we
conclude that $\beta\in K$, which implies $(c)$, since $\gamma_1$ is a
generator.
\end{proof}

\begin{remark}\label{Rmumforddescent}
  It would be interesting to determine whether the implication $(a)
  \implies (c)$ in Proposition~\ref{Pdescentcriterion} holds for
  all Mumford curves. Here is a general framework one could use.  Let $\Gamma$ be a Schottky group, and let $X$, $Y$, $L$, and $K$ be as
  in Proposition~\ref{Pdescentcriterion}(a).  Write $N(\Gamma)$ for
  the normalizer of $\Gamma$ in $PGL_2(L)$, and write $N^{sl}(\Gamma)$
  for the normalizer of $\Gamma$ in $PGL_2(L) \rtimes \Gal(L/K)$, with
  the standard action for the semidirect product --- this latter group is
  the set of $(L/K)$-semilinear actions on $\proj^1_L$.\footnote{As an aside, one can show
  that $N^{sl}(\Gamma)$ consists of the ordered pairs $(\gamma,
  \sigma)$ such that $\gamma \Gamma \gamma^{-1} = \sigma^{-1} \Gamma
  \sigma$.
  That is, conjugation of $\Gamma$ by $\gamma$ yields the
  same group as acting on all of the entries of the matrices in
  $\Gamma$ by $\sigma$.}

  We then have the
  following pushout diagram of exact sequences of groups:

\begin{equation}\label{Eliftingdiagram}
   \xymatrixrowsep{0.8cm}
   \xymatrixcolsep{2cm}
    \xymatrix{
      1 \ar[r] & N(\Gamma) \ar[d] \ar[r] & N^{sl}(\Gamma) \ar[d] \ar[r]
     & \Gal(L/K) \ar[r] \ar@{=}[d] & 1 \\
1 \ar[r] & N(\Gamma)/\Gamma  \ar[r]  & N^{sl}(\Gamma)/\Gamma \ar[r]
& \Gal(L/K) \ar[r] & 1}
\end{equation}
Since $Y = \Gamma \backslash (\proj^{1, \an}_L \setminus \mc{L})$, 
Proposition~\ref{prop:generallifting}
and (\ref{diagram:unicover}) show that 
 any $L$-semilinear automorphism of $X_L \cong Y$
lifts to an $L$-semilinear automorphism of $\proj^{1, \an}_L \setminus
\mc{L}$, and this lift is unique up to translation by $\Gamma$.  So
the semilinear $\Gal(L/K)$-action on $Y$ lifts to a semilinear
$\Gal(L/K)$-action on $\proj^{1, \an}_L \setminus \mc{L}$ up to
translation by $\Gamma$, which in
turn lifts to a semilinear $\Gal(L/K)$-action on $\proj^1_L$ up to
translation by $\Gamma$ using
Proposition~\ref{prop:extensionofliftings}.  In other words, if part
$(a)$ of Proposition~\ref{Pdescentcriterion} holds, then the bottom exact sequence of (\ref{Eliftingdiagram})
admits a \emph{section} from $\Gal(L/K)$.  If we could show that the
\emph{top} exact sequence of (\ref{Eliftingdiagram}) were to admit a
section from $\Gal(L/K)$, then an application of Hilbert's Theorem 90
(see Lemma~\ref{Lhilb90} below) would show that, up to a change of
variables, the $\Gal(L/K)$-action on $Y$ comes from descending
the \emph{purely arithmetic} $\Gal(L/K)$-action on $\proj^{1, \an}_L
\setminus \mc{L}$.\footnote{In the proof of $(a) \implies (c)$ in
  Proposition~\ref{Pdescentcriterion}, the fact that $\sigma(\beta) =
  \beta$ shows that a section of the top exact sequence can be given explicitly
  by $\sigma \mapsto (1, \sigma) \in N^{sl}(\Gamma)$ as in
  the previous footnote.}   The statement that this action descends is
precisely part $(b)$ of Proposition~\ref{Pdescentcriterion}.
\end{remark}

\begin{remark}
A Mumford curve over $L$ descends to a Mumford curve over $K$
if and only if its Schottky group is defined over $K$ (after possibly
changing a parameter).  Our examples in \S\ref{Sexample} are Mumford
curves that descend to curves over $K$, but not to Mumford curves.
\end{remark}

\begin{remark}\label{RKforms}

Proposition~\ref{Pdescentcriterion} can be applied to study curves that are isomorphic to Mumford curves after base change, as follows.
Let $\Gamma \subset \PGL_2(L)$ be a Schottky group satisfying all the conditions of Proposition \ref{Pdescentcriterion}, let $Y$ be the associated Mumford curve over $L$, and let $X$ be a curve over $K$ such that $X_L \cong Y$.
Then, the Galois cohomology set $H^1(\Gal(L/K), \Aut(X_L))$ classifies
all isomorphy classes of $K$-curves that become isomorphic to $Y$
after base change to $L$, also called $K$-\emph{forms} of $Y$.
In general, it is not easy to compute $H^1(\Gal(L/K), \Aut(X_L))$, but one can use the uniformization property to find some of its elements, as follows.
For every element $\gamma \in PGL_2(L)$, the group $\Gamma' := \gamma^{-1}\Gamma \gamma$ is again a Schottky group.
If $\Gamma'$ satisfies the condition $\sigma^{-1}\Gamma'\sigma = \Gamma'$, then the arithmetic action of $\Gal(L/K)$ on $\P^{1}_L$ descends to an action on $Y$, as in the proof of part $(b) \implies (a)$ of Proposition \ref{Pdescentcriterion}.
As such, it defines a curve $X'$ over $K$, which is not isomorphic to $X$, in general.
More in detail, the automorphism $\gamma$ induces a cocycle $\xi_\gamma \in Z^1(\Gal(L/K), PGL_2(L))$ via the formula $\xi_\gamma(\sigma) = \gamma^\sigma \gamma^{-1}$.
Since we have that 
\[\gamma^\sigma \gamma^{-1} \Gamma (\gamma^\sigma \gamma^{-1})^{-1} = \gamma^\sigma \Gamma' {(\gamma^\sigma)}^{-1} = \Gamma,\]
using the fact that both $\Gamma$ and $\Gamma'$ are $\Gal(L/K)$-equivariant, the automorphism $\xi_\gamma(\sigma)$ normalizes $\Gamma$ and descends to an automorphism $\xi(\sigma)$ of $X_L$.
This process generates a cocycle $\xi \in Z^1(\Gal(L/K), \Aut(X_L))$, whose cohomology class $[\xi] \in H^1(\Gal(L/K), \Aut(X_L))$ can be nontrivial, even though $\xi_\gamma$ is necessarily cohomologous to the trivial cocycle by Hilbert's Theorem 90.
The forms of a Mumford curve arising in this way are said to be \emph{arithmetic}. 
Examples of nontrivial arithmetic forms of Mumford curves will be
given in detail in \S\ref{Sexample}. We note that not all $K$-forms of
Mumford curves are arithmetic.  For example, non-trivial principal homogeneous spaces for Tate curves (see e.g. \cite[\S X.3]{Silverman16}) are twists of Tate curves but correspond to cocycles that do not lift to $Z^1(\Gal(L/K), PGL_2(L))$.
\end{remark}

\subsection{Fields of definition of type 2 points}\label{Sfieldofdefinition}

\begin{defn}\label{def:defined}
Let $L/K$ be a finite Galois extension. Let us fix a parameter on $\proj^{1, \mathrm{an}}_L$, so that the natural arithmetic action of $\Gal(L/K)$ is well defined.
If $x$ is a point of $\proj^{1, \mathrm{an}}_L$ of type 2, we say that
$x$ is \emph{defined over} $K$ if it is of the form $\eta_{a,r}$ for
$a\in K$ and $r\in |K^\times|$ (see Remark~\ref{Dc}).
We can extend this definition to Mumford curves: if $Y$ is a Mumford
curve over $L$, and we denote by $p \colon \proj^{1,\mathrm{an}}_L
\setminus \mc{L} \to Y^{\mathrm{an}}$ the universal cover coming from
Schottky uniformization, a point $y\in Y^{\mathrm{an}}$ of type 2 is
\emph{defined over} $K$ if there exists a point in the pre-image $p^{-1}(y)$ that is defined over $K$.
Note that this notion depends on the choice of a Schottky group associated with $Y$.
\end{defn}

\begin{remark}\label{Rmultiplicity1}
A point of $Y$ defined over any intermediate extension $M$ of $L/K$ has
multiplicity 1. 
It is sufficient to prove this for a type 2 point $x$ of $\proj^{1,\an}_L$, as the projection $p$ is a local isomorphism and the multiplicity of a point only depends on the analytic structure.
To do this, consider the special fiber of the model of $\proj^1_M$
corresponding to the singleton $\{x\}$.
It consists of a unique component $\mc{X}_k$, whose closed points are smooth if, and only if, their formal fibers are Berkovich discs centered in $M$-rational points with $M$-rational radius. 
Since this is the case when $x$ is defined over $M$, we have that the curve $\mc{X}_k$ is reduced, which is equivalent to the fact that $x$ has multiplicity 1.
\end{remark}


Let $X$ be an arithmetic $K$-form of a Mumford curve $Y$ defined over
$L$ as in Remark~\ref{RKforms}.
Then, there is a Schottky group $\Gamma$ associated with $X_L$ such
that the action of $\Gal(L/K)$ on $X_L^{\mathrm{an}}$ inherited from
the base change is compatible with the arithmetic action on $\proj^{1,\mathrm{an}}_L$ via the uniformization map $p\colon \proj^{1,\mathrm{an}}_L \setminus \mc{L} \to X_L^{\mathrm{an}}$.
For the rest of this section, we deal only with arithmetic forms and we fix a $\Gamma$ as above. In this way, we can write about points defined over subfields of $L$ (in the sense of Definition \ref{def:defined}) without ambiguity.

\begin{lemma}\label{lem:smoothdescent}
Let $X$ be an arithmetic $K$-form of a Mumford $L$-curve, and let $x \in X_L^{\mathrm{an}}$ be a point of type 2 defined over a field $M$ with $K\subseteq M \subseteq L$.
Let $\mc{X}$ be a formal model of $X_L^{\mathrm{an}}$ on which the
action of $\Gal(L/M)$ is well-defined, and whose associated set of type 2 points contains $x$. Denote by $\ol{V}$ the component of $\mc{X}_k$ corresponding to $x$ and by $P$ a smooth point of $\ol{V}$.
Then, the Galois group $\Gal(L/M)$ acting on $\mc{X}$ fixes $\ol{V}$ pointwise, the image of $P$ in the quotient model $\mc{X} / \Gal(L/M)$ is a smooth point, and the image of $\ol{V}$ has multiplicity one.
\end{lemma}
\begin{proof}
Let $\pi \colon X^{\mathrm{an}}_L\to X^{\mathrm{an}}_M$ and
$\tilde{\pi} \colon \proj^{1,\mathrm{an}}_L\to\proj^{1,\mathrm{an}}_M$ be the canonical projection maps.
The fact that the point $x$ is defined over $M$ yields the following
two consequences: 

\begin{enumerate}
\item There is a point $y \in p^{-1}(x)$ defined over $M$.  In
  particular, $y$ is fixed by every element of $\Gal(L/M)$;
\item The multiplicity of $\tilde{\pi}(y)$ is equal to 1 (Remark~\ref{Rmultiplicity1}).
\end{enumerate}
Since the action of $\Gal(L/K)$ on $X_L$ descends from the purely arithmetic action on $\proj^{1,\mathrm{an}}_L$, property $(a)$ results in the fact that $\ol{V}$ is fixed by $\Gal(L/M)$.
Moreover, since $L/M$ is totally ramified and $y$ is defined over $M$, the group $\Gal(L/M)$ acts with full inertia on the completed residue field $\scrH(y)$.
As a result, the action of $\Gal(L/M)$ on the component $\ol{V}$ is trivial and the image of the smooth point $P \in \ol{V}$ in the quotient model $\mc{X} / \Gal(L/M)$ is a smooth point.
Finally, property $(b)$ ensures that the image of $\ol{V}$ in the quotient model $\mc{X} / \Gal(L/M)$ has multiplicity 1.
\end{proof}

\begin{lemma}\label{lem:multiplicities}
Let $X$ be an arithmetic $K$-form of a Mumford curve over $L$, let $\pi \colon X_L^{\mathrm{an}} \to
X^{\mathrm{an}}$ be the canonical projection induced by the base change, and let $x \in X_L^{\mathrm{an}}$ be a point of type 2 defined over a field $M$ with
$K \subseteq M \subseteq L$.
Then we have that the multiplicity $m(\pi(x))$ divides the degree $[M:K]$.
\end{lemma}
\begin{proof}
We have that the multiplicity of a type 2 point $y$ in a $K$-analytic curve is given by the index $\big[|\mathscr H(y)^\times|:|K^\times|\big]$ \cite[Proposition 6.5.2 (1)]{Ducros14}.
Consider now the projection morphism $\pi' \colon~X_L^\mathrm{an} \to X_M^{\mathrm{an}}$. 
Since $x$ is defined over $M$, we can apply Lemma
\ref{lem:smoothdescent} to the model corresponding to the set $\{x\}$ to show that the point $\pi'(x)\in X_M^\mathrm{an}$ has multiplicity 1,
that is to say $|\mathscr H(\pi'(x))^\times|=|M^\times|$, and hence $|\mathscr H(\pi(x))^\times| \subset |M^\times|$. This implies that \[m(\pi(x))=\big[|\mathscr H(\pi(x))^\times|:|K^\times|\big] = \frac{\big[|M^\times|:|K^\times|\big]}{\big[|M^\times|:|\mathscr H(\pi(x))^\times|\big]},\]
which proves the statement.
\end{proof}

\section{Proof of Theorem~\ref{Tnegative}}\label{Sexample}

\subsection{Generalities}\label{Sgeneralities}
In this subsection, we set up our examples for
Theorem~\ref{Tnegative}.  Our notation is as follows:
$L/K$ is a finite Galois extension, and $\Gamma \subseteq PGL_2(L)$ is
a Schottky group of rank $g \geq 2$ with generators $\gamma_1, \ldots,
\gamma_g$ and associated Schottky figure $$(D^{\pm}(\gamma_1),
D^{\pm}(\gamma_1^{-1}), \ldots, D^{\pm}(\gamma_g),
D^{\pm}(\gamma_g^{-1}))$$ (Definition~\ref{def:Schottkyfigure}).  
Write $\mf{D} = \proj^{1,\an}_L \setminus \mc{L}$ where $\mc{L}$
is the limit set of $\Gamma$, let $Y^{\an} = \Gamma \backslash \mf{D}$, and let $Y$ be the
corresponding projective curve over $L$, which is the genus $g$
Mumford curve corresponding to $\Gamma$.  We assume that $\Gamma$
satisfies part $(b)$ of Proposition \ref{Pdescentcriterion}, which by part $(a)$ of the
same proposition shows that $Y$ descends to a $K$-curve $X$.  The
curve $X$ is an arithmetic $K$-form of $Y$, in the sense of \S\ref{Sdescent}.

Let $\mc{Y}^{\stab}$ be the stable model of $Y$, and let $\mc{Y}'$
be the semi-stable model of $Y$ associated to the set $S_Y$ of
multiplicity 1 points in the skeleton $\Sigma_Y$ of $Y^{\an}$ as in
Proposition~\ref{prop:mult1}.  Recall that the set of irreducible components
of the special fiber of $\mc{Y}'$ is in natural one-to-one
correspondence with $S_Y$.

Let $\varpi' \colon \mc{D}' \to \hat{\mc{Y}}'$ be the morphism of formal
schemes coming from the uniformization map $\mf{D} \to Y^{\an}$ as in Lemma~\ref{lem:formallift}.
Recall that $\mc{D}'$ is a formal model of
$\mf{D}$.  We write $\Sigma_{\mf{D}}$ for the skeleton of $\mf{D}$
(Theorem~\ref{thm:unif}), and we write $S_{\mf{D}}$ for the set of
multiplicity $1$ points on $\Sigma_{\mf{D}}$.  

Let $F \subseteq \mf{D}$ be a
fundamental domain for $\Gamma$ associated to the Schottky figure
above, as in Notation~\ref{NSchottky}.  Write $\Sigma_{F}$ for its skeleton, and
write $S_F \subseteq S_{\mf{D}}$ for the set of multiplicity $1$ points on $\Sigma_F$.
The group $\Gamma$ acts on $S_{\mf{D}}$.
By the definition of a fundamental domain, no two
elements of $S_{F}$ lie in a common $\Gamma$-orbit of $S_{\mf{D}}$.
So the map $S_{F} \to S_Y := \Gamma \backslash
S_{\mf{D}}$ induced by
$S_{F} \hookrightarrow S_{\mf{D}} \to S_Y$ is injective, and Theorem~\ref{thm:unif} shows it is in fact bijective.

The group $\Gal(L/K)$ acts on $\Sigma_Y$ via descent of its action on $\Sigma_{\mf{D}}$, and
thus it acts on $\mc{Y}'$.  We pick a semi-stable blow-down
$\mc{Y}''$ of $\mc{Y}'$;
in the examples,
$\mc{Y}''$ will be chosen strategically.  We make the further
assumption that the action of $\Gal(L/K)$ on $Y$ extends to
$\mc{Y}''$.  Let $\mc{D}'' \to \hat{\mc{Y}}''$ be the morphism of
formal schemes coming from the uniformization map $\mf{D} \to Y^{\rm an}$
as in Lemma~\ref{lem:formallift}.
The quotient $\mc{Y}'' / \Gal(L/K)$ 
is a normal model $\mc{X}$ of $X$ with special fiber
$\mc{X}_k$.

\begin{lemma}\label{LtransfertoY}
  Let $x \in \mc{X}_k \subseteq \mc{X}$, let $y$ be a preimage of $x$ under
  $\mc{Y}'' \to \mc{X}$, and let $\tilde{y}$ be the unique preimage of
  $y$ under $\mc{D}'' \to \hat{\mc{Y}}'' \to \mc{Y}''$ lying in the
  reduction $\ol{F}$ of the fundamental domain $F$.
Let $H \subseteq \Gal(L/K)$ be any subgroup such that
  $H\tilde{y} \subseteq \ol{F}$. 
Write $\mc{Z} = \mc{D}'' / H$ using the standard $\Gal(L/K)$-action, let $z$ be the image of
$\tilde{y}$ in $\mc{Z}$, and let $y^H$ be the image of $y$ in
$\mc{Y}''/H$, as in the diagram below.

   $$
   \xymatrixrowsep{0.8cm}
   \xymatrixcolsep{2cm}
    \xymatrix{
   \ \tilde{y}  \in \mc{D}'' \ar[d]_{\Gamma} \ar[rr]^{H} & &
   \mc{Z} \ni z \\
     y \in \hat{\mc{Y}}'' \ar[r] & \mc{Y}'' \ar[r]^-{H} & \mc{Y}'' / H \ni y^H
    }
    $$
Then $\hat{\mc{O}}_{\mc{Y}''/H,\, y^H} \cong
\hat{\mc{O}}_{\mc{Z}, z}$.
In particular, if $\Gal(L/K) \tilde{y} \subseteq \ol{F}$, then taking
$H = \Gal(L/K)$ gives
$\hat{\mc{O}}_{\mc{X}, x} \cong \hat{\mc{O}}_{\mc{Z}, z}$.

%
\end{lemma}

\begin{proof}
By Proposition~\ref{prop:semistable}, $\hat{\mc{O}}_{\mc{Y}'',
  \tilde{y}} \cong \hat{\mc{O}}_{\mc{D}'', \tilde{y}} $.  The $\Gal(L/K)$-action on $\mc{Y}''$ comes from the
$\Gal(L/K)$-action on $\mc{D}''$ by descent, thus the same is true for
the $H$-actions.   By assumption, the $H$-orbit of $\tilde{y}$ lies in the reduction $\ol{F}$ of $F$.  Since $F \to Y^{\an}$ is
an isomorphism onto its image, the formal
neighborhood $U$ of the $H$-orbit of $\tilde{y}$ maps isomorphically onto its image in
$\mc{Y}''$, and this map respects the $H$-action.  The lemma follows.
\end{proof}

\begin{prop}\label{Pmult1to9}
  Let $x \in \mc{X}_k$, and maintain the notation of Lemma~\ref{LtransfertoY}.  Assume the
  $\Gal(L/K)$-orbit of $\tilde{y}$ lies in $\ol{F}$.  Let $V$ be
  an irreducible component of $\mc{X}_k$ containing $x$, let
  $V''$ be the irreducible component of $\mc{D}''$ in the
  preimage of $V$ lying in $\ol{F}$, and let
  $W$ be the image of $V''$ in $\mc{Z}$.  Then $m_V =
  m_W$.
\end{prop}

\begin{proof}
If $z$ is as in Lemma~\ref{LtransfertoY}, then by that lemma,
$\hat{\mc{O}}_{\mc{X}, x} \cong \hat{\mc{O}}_{\mc{Z},z}$. 
Since $x \in V$ and $z \in W$, and the complete local rings of $x$ and
$z$ detect the multiplicities of the irreducible components they lie
on, we conclude that $m_V = m_W$.
\end{proof}

Let $\mc{X}^{\reg} \to \mc{X}$ be the minimal regular snc-resolution of
$\mc{X}$.  Write $\mc{X}^{\reg}_k$ for the special fiber of $\mc{X}^{\reg}$.

\begin{prop}\label{Pexceptionalmultiplicity1}
  Let $x \in \mc{X}_k$, and maintain the notation of Lemma~\ref{LtransfertoY}.  Assume the
  $\Gal(L/K)$-orbit of $\tilde{y}$ lies in $\ol{F}$.  Let $d$ be the
  $\lcm$ of the multiplicities of the irreducible components of
  $\mc{X}_k$ passing through $x$.  If $E$
  is a principal component of $\mc{X}^{\reg}_k$ that lies on the
  exceptional divisor of $\mc{X}^{\reg} \to \mc{X}$ and whose
  image is $x$, then
  $m_E$ divides $d$.
\end{prop}

\begin{proof}
By Lemma~\ref{LtransfertoY}, $\hat{\mc{O}}_{\mc{X}, x} \cong
\hat{\mc{O}}_{\mc{Z}, z}$.   By Proposition~\ref{Pmult1to9}, $d$ is the $\lcm$ of the
 multiplicities of the irreducible components of the special fiber of $\mc{Z}$ passing
 through $z$.  Let $S$ be the
set of Mac Lane valuations corresponding to these
components, and let $\mc{Z}_S$ be the normal model of $\proj^1_K$
corresponding to $S$.  By Remark~\ref{Lsamelocalring}, we have $\hat{\mc{O}}_{\mc{Z}, z}
\cong \hat{\mc{O}}_{\mc{Z}_S, \phi(z)}$, where $\phi$ is the injection in
Remark~\ref{Lsamelocalring}.  So $d$ is also the $\lcm$ of the
multiplicities of the irreducible components of the special fiber of
$\mc{Z}_S$ passing through $\phi(z)$.
Thus it suffices
to show that any principal component of the exceptional divisor above
$\phi(z)$ on the minimal regular resolution of $\mc{Z}_S$ has multiplicity
dividing $d$.  This follows from
Proposition~\ref{Pmaclaneresolution}.
\end{proof}

\begin{prop}\label{Pexceptionalmultiplicity2}
  Let $x \in \mc{X}_k$, and maintain the notation of
  Proposition~\ref{Pmult1to9}.  Assume that $x$ lies on only one
  irreducible component $V$ of $\mc{X}_k$, and that the
  corresponding type 2 point
  is defined over a field
  $M$ such that $K \subseteq M \subseteq L$ and $\Gal(M/K) \cong
  \ints/p$.  Assume further that $\Gal(M/K)$ acts nontrivially on the
  preimage of $V$ in $\mc{Y}''/\Gal(L/M)$.  If $E$
  is a principal component of $\mc{X}^{\reg}_k$ that lies on the
  exceptional divisor of $\mc{X}^{\reg} \to \mc{X}$ and whose
  image is $x$, then
  $m_E$ divides $p$.
\end{prop}

\begin{proof}
Let $H = \Gal(L/M)$.  By Lemma~\ref{lem:smoothdescent}, the preimage
$y^H$ of $x$ in
$\mc{Y}''/H$ is smooth. We may assume $x$
is singular in $\mc{X}$ (otherwise $E$ would not exist).  Since $x$
is the image of the smooth point $y^H$ after taking the quotient by the action
of $\Gal(M/K)$ on $\mc{Y}''/H$, we have that $\Gal(M/K)$ fixes
$y^H$.  Since the irreducible component of the special fiber of
$\mc{Y}''/H$ on which $y^H$ lies has genus $0$, and $\Gal(M/K)$ acts
nontrivially on this component, 
the Hurwitz formula for wildly ramified curves shows that the induced
$\Gal(M/K)$-action on the special fiber of $\Spec (\hat{\mc{O}}_{\mc{Y}''/H,\,  y^H})$ has
ramification jump $1$. 
This means that $x$ is a \emph{weak wild arithmetic quotient
  singularity} with stabilizer $\ints/p$ in the language of
\cite[Definition~3.7]{ObusWewers}.  By \cite[Corollary 6.5, Definition
6.6, and Corollary 7.13(i)]{ObusWewers}, the unique principal
component of the exceptional divisor of the minimal resolution of the
singularity at $x$ has multiplicity $p$, proving the proposition.
\end{proof}

\subsection{An example ($g=2, p=2$)}
Let $k$ be an algebraically closed field of characteristic $2$, and
let $K = \Frac(W(k))$.  Let $L = K[a, b]/(a^2 - 2, b^2 + 1)$.  Write
$\sqrt{2}$ and $i$
for the respective images of $a$ and $b$ in $L$.  The group $\Gal(L/K) \cong \ints/2
\times \ints/2$ is generated by $\sigma$ and $\tau$ acting nontrivially on 
$\sqrt{2}$ and $i$, respectively.  In this subsection, as a specific
example of the setup from \S\ref{Sgeneralities}, we construct a genus 2 curve $X/K$
with potentially multiplicative reduction, a $K$-rational point, and
stabilization index $e(X) = 2$, but for which $L/K$ (of degree 4) is the
minimal extension over which $X$ attains semi-stable reduction.
This proves the $p=2$ case of Theorem~\ref{Tnegative}.

Let $\Gamma$ be the
subgroup of $PGL_2(L)$ generated by
$$\gamma_1 := P A P^{-1}, \ \gamma_2 := QAQ^{-1},$$
where
$$P = \begin{bmatrix} -\sqrt{2} & -\sqrt{2} \\ -1 & 1 \end{bmatrix},\
Q = \begin{bmatrix} -1-i & 1-i \\ -1 & 1 \end{bmatrix},\
A = \begin{bmatrix} 2 & 0 \\ 0 & 1 \end{bmatrix}.
$$

\begin{prop}\label{PSchottky}\hfill
  \begin{enumerate}[\upshape (i)]
 \item  The group $\Gamma$ above is a Schottky group of rank 2, with Schottky figure
  given by
  $$(D^+(\gamma_1), D^+(\gamma_1^{-1}), D^+(\gamma_2),
  D^+(\gamma_2^{-1})) = (\ol{B}(\sqrt{2}, 2), \ol{B}(-\sqrt{2}, 2), \ol{B}(1 + i,
    3/2), \ol{B}(1-i , 3/2)),$$ and corresponding open subdisks
$$(D^-(\gamma_1), D^-(\gamma_1^{-1}), D^-(\gamma_2),
  D^-(\gamma_2^{-1})) =  (B(\sqrt{2}, 2), B(-\sqrt{2}, 2), B(1 + i,
    3/2), B(1-i , 3/2)),$$
    where $B(a, r)$ (resp.\ $\ol{B}(a,r)$) is the open (resp.\ closed)
    disk with center $a$ and radius $|2^r|$.
    \item If $Y^{\rm an} = \Gamma \backslash \mf{D}$ as in \S\ref{Sgeneralities} and $Y$ is the
      corresponding projective curve over $L$, then the standard action of
      $\Gal(L/K)$ on $\mf{D} \subseteq \proj^{1, \rm an}_L$ descends
      to $Y^{\rm an}$, and thus to $Y$.  In particular, $Y$ descends
      to a curve $X := Y/(\Gal(L/K))$ defined over $K$.
    \end{enumerate}
\end{prop}

\begin{proof}
  One computes straightforwardly that $\gamma(\proj^1_L \setminus
D^+(\gamma^{-1})) = D^-(\gamma)$ for
$\gamma \in \{\gamma_1^{\pm 1}, \gamma_2^{\pm 1}\}$.  Since each
$D^-(\gamma)$ is a maximal open subdisk of $D^+(\gamma)$, part (i) 
follows from Definition~\ref{def:Schottkyfigure}.

To prove part (ii), one first calculates explicitly that, as
automorphisms of $\proj^1_L$, we have $\sigma \gamma_1 \sigma^{-1} = \gamma_1^{-1}$ and $\tau
\gamma_2  \tau^{-1} = \gamma_2^{-1}$, whereas $\sigma$ commutes with
$\gamma_2$ and $\tau$ commutes with $\gamma_1$ (conceptually, this is
because the fixed points of $\gamma_1$, resp.\ $\gamma_2$, are $\pm
\sqrt{2}$, resp.\ $1 \pm i$, and $A$ is defined over $K$.  Thus conjugating by $\sigma$, resp.\ $\tau$, represents
reversing the direction of the ``flow'' of $\gamma_1$, resp.\ $\gamma_2$). 
So $\Gamma$ satisfies Proposition~\ref{Pdescentcriterion}(b), and thus $Y$ descends to a curve $X$
defined over $K$ such that $X \times_K L \cong Y$. In fact, since $\sigma$ and $\tau$ normalize
$\Gamma$ directly (without resorting to conjugation by an element of
$PGL_2(L)$), we have that $\Gal(L/K)$ acts on $Y$ via the
descent of its standard action on  $\proj^{1, \an}_L \setminus
\mc{L}$, and $X
= Y / \Gal(L/K)$.
\end{proof}

The curve $X = Y/(\Gal(L/K))$ is an arithmetic $K$-form of $Y$,
consistent with the notation in \S\ref{Sgeneralities}.

\begin{lemma}\label{Lrationalpoint}
The curve $X$ has a $K$-rational point.  In particular, it has index $1$.
\end{lemma}

\begin{proof}
Since the $L$-rational point $0 \in \aff^{1, \an}_L \subseteq \proj^{1, \an}_L$ does not lie in any of the
$D^-(\gamma)$, it does not lie in $\mc{L}$ (see
Notation~\ref{NSchottky}).  So $0 \in \mf{D}$, and its image
under $\mf{D} \to Y^{\an} \to X^{\an}$ is
a $K$-rational type 1 point of $X^{\an}$.  Thus $X$ has a $K$-rational point.
\end{proof}

Now, we define $\mc{Y}^{\stab}$, $\mc{Y}'$, 
$F$, $\ol{F}$,
$\Sigma_Y$,
$S_Y$,  $\Sigma_F$,
and $S_F$ as in \S\ref{Sgeneralities}. 
%
%
The skeleton $\Sigma_{F}$ is the skeleton
connecting the type 2 points $\eta_{\sqrt{2}, 2}$, $\eta_{1+i, 3/2}$, $\eta_{-\sqrt{2}, 2}$, and
$\eta_{1 - i, 3/2}$ of $\proj^{1, \an}_L$, but not including the latter two points.  We have
$$|\sqrt{2} - (-\sqrt{2})| = |2^{3/2}|,\ \ |(1 + i) - (1-i)| =
|2^1|,\ \ |\pm \sqrt{2} - (1 \pm i)| = |2^{3/4}|.$$  It follows that
the set $S_{F}$ of multiplicity $1$ points on $\Sigma_{F}$
consists of the points $\eta_{a, r}$, where the $(a,r)$ correspond to
the black dots in Figure~\ref{Fdual}.
%
The skeletons $\Sigma_F$ and $\Sigma_Y$ are pictured in
Figure~\ref{Fdual}.  Note that $\Sigma_Y$ is the dual graph of
the special fiber of $\mc{Y}'$. 

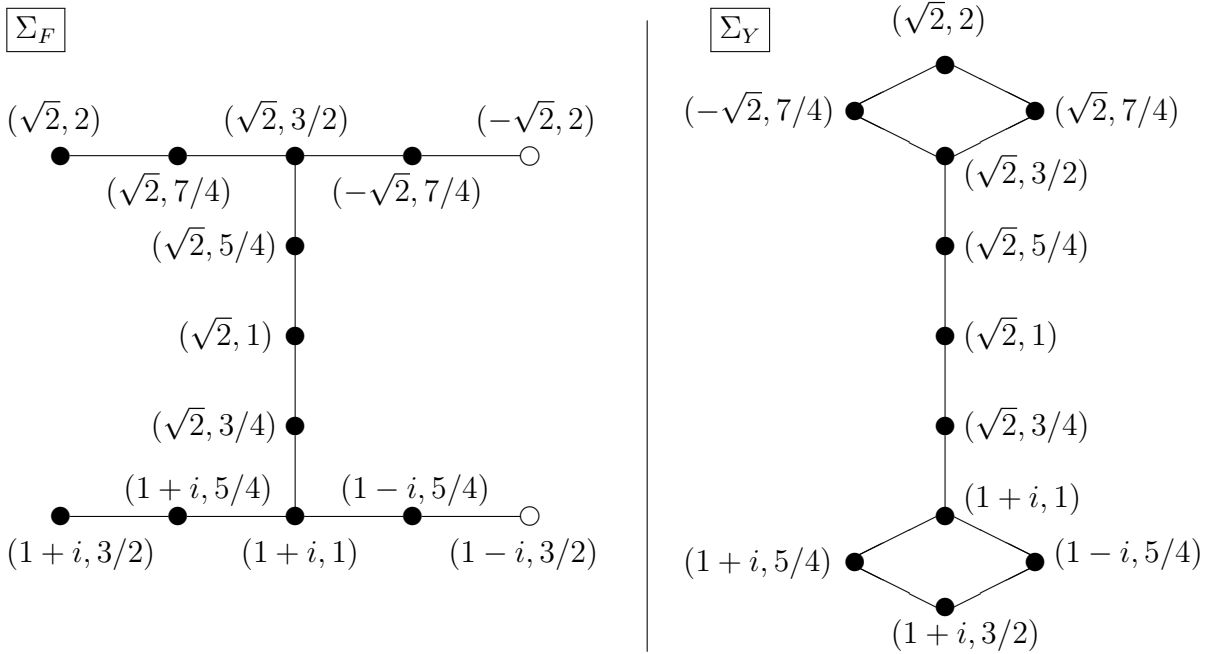
\begin{figure}
\begin{center}
\setlength{\unitlength}{1.2mm}
\begin{picture}(130,70)
\put(72, 0){\line(0,1){70}}
  \put(1, 68){$\boxed{\Sigma_F}$}
  \put(1, 58){$(\sqrt{2}, 2)$}
  \put(25, 58){$(\sqrt{2}, 3/2)$}
  \put(53, 58){$(-\sqrt{2}, 2)$}
  \put(12, 50){$(\sqrt{2}, 7/4)$}
  \put(37, 50){$(-\sqrt{2}, 7/4)$}
  \put(17, 44){$(\sqrt{2}, 5/4)$}
  \put(20, 34){$(\sqrt{2}, 1)$}
  \put(17, 24){$(\sqrt{2}, 3/4)$}
  \put(1, 10){$(1+i, 3/2)$}
  \put(27, 10){$(1+i,1)$}
  \put(38, 17){$(1-i,5/4)$}
  \put(14, 17){$(1+i,5/4)$}
  \put(50, 10){$(1-i, 3/2)$}

  \put(7,55){\circle*{2}}
\put(20,55){\circle*{2}}
\put(33,55){\circle*{2}}
\put(46,55){\circle*{2}}
\put(59,55){\circle{2}}
\put(33,45){\circle*{2}}
\put(33,35){\circle*{2}}
\put(33,25){\circle*{2}}
\put(33,15){\circle*{2}}
\put(7,15){\circle*{2}}
\put(20,15){\circle*{2}}
\put(46,15){\circle*{2}}
\put(59,15){\circle{2}}

\put(8, 55){\line(1,0){11}}
\put(21, 55){\line(1,0){11}}
\put(34, 55){\line(1,0){11}}
\put(47, 55){\line(1,0){11}}
\put(8, 15){\line(1,0){11}}
\put(21, 15){\line(1,0){11}}
\put(34, 15){\line(1,0){11}}
\put(47, 15){\line(1,0){11}}
\put(33, 54){\line(0,-1){8}}
\put(33, 44){\line(0,-1){8}}
\put(33, 34){\line(0,-1){8}}
\put(33, 24){\line(0,-1){8}}

\put(79, 68){$\boxed{\Sigma_Y}$}
\put(99, 69){($\sqrt{2}, 2)$}
\put(117, 59){($\sqrt{2}, 7/4)$}
  \put(107, 52){$(\sqrt{2}, 3/2)$}
  \put(76, 59){$(-\sqrt{2}, 7/4)$}
  \put(107, 44){$(\sqrt{2}, 5/4)$}
  \put(107, 34){$(\sqrt{2}, 1)$}
  \put(107, 24){$(\sqrt{2}, 3/4)$}
  \put(99, 1){$(1+i, 3/2)$}
  \put(107, 16){$(1+i,1)$}
  \put(117, 10){$(1-i,5/4)$}
  \put(76, 9){$(1+i,5/4)$}

\put(105,65){\circle*{2}}
\put(105,55){\circle*{2}}
\put(105,45){\circle*{2}}
\put(105,35){\circle*{2}}
\put(105,25){\circle*{2}}
\put(105,15){\circle*{2}}
\put(105, 5){\circle*{2}}
\put(95, 60){\circle*{2}}
\put(115, 60){\circle*{2}}
\put(95, 10){\circle*{2}}
\put(115, 10){\circle*{2}}

\put(104, 65){\line(-2, -1){9}}
\put(106, 65){\line(2, -1){9}}
\put(104, 55){\line(-2, 1){9}}
\put(106, 55){\line(2, 1){9}}
\put(105, 54){\line(0, -1){8}}
\put(105, 44){\line(0, -1){8}}
\put(105, 34){\line(0, -1){8}}
\put(105, 24){\line(0, -1){8}}
\put(104, 15){\line(-2, -1){9}}
\put(106, 15){\line(2, -1){9}}
\put(104, 5){\line(-2, 1){9}}
\put(106, 5){\line(2, 1){9}}

\end{picture}
\end{center}
\caption{The skeletons $\Sigma_F$  and $\Sigma_Y$.  Black dots correspond to the points of $S_F$ and $S_Y$,
  respectively.  White dots are on the boundary of $\Sigma_F$ but not included.  In $\Sigma_F$, a dot labeled $(a, r)$ corresponds to
  the point $\eta_{a,r}$.  In $\Sigma_Y$, a dot labeled $(a,r)$ is
  the image of the corresponding dot on $\Sigma_F$.}\label{Fdual}
\end{figure}

\begin{prop}\label{Pminfield}
The extension $L/K$ is the minimal extension over which $X$ acquires
semi-stable reduction.
\end{prop}

\begin{proof}
By construction, $Y \cong X \times_K L$ has semi-stable reduction.  If $X$ attained
semi-stable reduction over a Galois subextension $K \subseteq M
\subsetneq L$, then the $\Gal(L/K)$-action on the special fiber
$\mc{Y}^{\stab}_k$ of the stable model $\mc{Y}^{\stab}$ of $Y$
would factor through $\Gal(M/K)$
(see Remark~\ref{Rfaithful}), so it suffices to show that
$\Gal(L/K)$ acts faithfully on $\mc{Y}^{\stab}_k$.  From the
right-hand part of Figure~\ref{Fdual}, one sees that
$\mc{Y}^{\stab}$ is constructed from $\mc{Y}'$ by contracting all
components of the special fiber of $\mc{Y}'$ except those
corresponding to the images of $\eta_{\sqrt{2}, 3/2}$ and $\eta_{1 +i ,
1}$.  But $\sigma$ acts nontrivially on the first component (since it
switches $\eta_{\pm\sqrt{2}, 7/4}$), and $\tau$ acts nontrivially on
the second component (since it switches $\eta_{1 \pm i, 5/4}$).  So
$\Gal(L/K)$ acts faithfully on $\mc{Y}^{\stab}_k$.  
\end{proof}

Now we define $\mc{Y}''$ by letting $\mc{Y}' \to \mc{Y}''$ be the morphism
given by blowing down the components corresponding to the vertices
labeled $(\pm \sqrt{2}, 7/4)$, $(1 \pm i,
5/4)$, $(\sqrt{2}, 3/4)$, and $(\sqrt{2}, 5/4)$ in the right-hand part
of Figure~\ref{Fdual}.
Then $\mc{Y}''$ is also a semi-stable model of $X_L$ on which
$\Gal(L/K)$ acts, consistent with the assumptions on $\mc{Y}''$ in
\S\ref{Sgeneralities}.  As in \S\ref{Sgeneralities}, we set $\mc{X}:=~
\mc{Y}''/(\Gal(L/K))$ and $\mc{X}_k$ is its special fiber.

\begin{lemma}\label{Lallcomponents}
  All components of $\mc{X}_k$ have multiplicity dividing $2$.
\end{lemma}

\begin{proof}
 Our construction of $\mc{Y}''$ blew down all the components with
 minimal field of definition $L$. So the type 2 points corresponding to the irreducible
 components of the special fiber of $\mc{Y}''$ are each defined over
 either $K(\sqrt{2})$ or $K(i)$.  By Lemma~\ref{lem:multiplicities},
 the images of these components in $\mc{X}$ have multiplicity dividing 2.
\end{proof}

Write $V_1, V_2$ for the irreducible
components of $\mc{X}_k$ corresponding to the vertices
$(\sqrt{2}, 2)$, $(1+i, 3/2)$ in Figure~\ref{Fdual}, respectively.
Write $U$ for the union of all of the other irreducible
components of $\mc{X}_k$. 

As in \S\ref{Sgeneralities}, let $pr \colon \mc{X}^{\reg} \to \mc{X}$ be the minimal regular snc-resolution of
$\mc{X}$.  Write $\mc{X}^{\reg}_k$ for the special fiber of
$\mc{X}^{\reg}$.  Recall also from \S\ref{Sgeneralities} that we
have the cover $\mc{D}'' \to \mc{Y}''$.

\begin{prop}\label{Pweakwild}
  If $E$ is a principal component of $\mc{X}^{\reg}_k$ that is
  exceptional for $pr \colon \mc{X}^{\reg} \to \mc{X}$, then the multiplicity of
  $E$ divides $2$.
\end{prop}
  

\begin{proof}
%
Let $x$ be the image of $E$ in $\mc{X}$.  First assume $x \in
U$.  Then, if $\tilde{y} \in \ol{F} \subseteq \mc{D}''$ lies above $x$ as in
Lemma~\ref{LtransfertoY}, we have $\Gal(L/K)\tilde{y}$ lies in $\ol{F}$, so by
Proposition~\ref{Pexceptionalmultiplicity1} and
Lemma~\ref{Lallcomponents}, the multiplicity of $E$ divides $2$.

Now, assume $x \notin U$.  Then we are in the situation of
Proposition~\ref{Pexceptionalmultiplicity2}, with $M = K(\sqrt{2})$ or
$M = K(i)$, depending on whether $x \in V_1$ or $V_2$.  Observe that,
in the first case,
$\Gal(M/K)$ acts non-trivially on $\mc{Y}''/\langle \tau
\rangle$ above $V_1$, since it switches $\eta_{\pm \sqrt{2},
  7/4}$, and in the second case, $\Gal(M/K)$ acts nontrivially on $\mc{Y}''/\langle \sigma
\rangle$ above $V_2$, since it switches $\eta_{1 \pm i, 5/4}$.  So
the proposition follows from
Proposition~\ref{Pexceptionalmultiplicity2}. 
\end{proof}

\begin{corollary}\label{Cexceptional}
Every principal component of $\mc{X}^{\reg}_k$ has
multiplicity dividing $2$.
\end{corollary}

\begin{proof}
  This follows immediately from Lemma~\ref{Lallcomponents} and
  Proposition~\ref{Pweakwild}.
\end{proof}

%

In particular, $e(\mc{X}^{\reg})$ divides $2$ (in fact, it equals $2$).  
As mentioned in \cite[top of p.\
46]{HalleNicaise}, this implies that the stabilization index $e(X)$
divides $2$.  Since $[L:K] = 4$, this completes the example.

\subsection{A class of examples ($g=2p$, any $p$)}

Let $k$ be an algebraically closed field of characteristic $p$.  We
fix a $p^2$-th root of unity $\zeta$ in $\overline{\Frac(W(k))}$, we let $K=\Frac(W(k))(\zeta^p)$, and we fix a uniformizer $\pi_K$ of $K$.
Moreover, we choose an element $s\in K$ such that $|s|=1$ and $|s-1|=1$ and we set $\alpha= s+\sqrt[p]{\pi_K}$. 
We consider the finite Galois extension $L=K(\zeta, \alpha)$ of $K$.
Since $L$ is obtained from $K$ by adjoining $p$-th roots, we can use Kummer theory to establish that the Galois group $\Gal(L/K)$ is isomorphic to $\Z/p\Z \times \Z/p\Z$, generated by automorphisms $\sigma$ and $\tau$ such that 
\[\begin{cases} \sigma(\zeta) = \zeta^{p+1}, & \tau(\zeta)=\zeta, \\ \sigma(\alpha)=\alpha, & \tau(\alpha) = s+\zeta^p \sqrt[p]{\pi_K}.\end{cases}\] 
In the rest of the section, we construct a curve $X$ over $K$ with potentially multiplicative reduction, a $K$-rational point, stabilization index $e(X) = p$, and such that $L/K$ (of degree $p^2$) is the minimal extension over which $X$ acquires semi-stable reduction.
This completes the proof of Theorem~\ref{Tnegative}.

Let $\beta, \beta' \in K$ be such that $|\beta|, |\beta'|\leq|\pi_K|^4$. 
For $i=0,\dots, p-1$ we define $A_i$ to be the following element of
$PGL_2(L)$:

\[  A_{i}:= \begin{bmatrix} (1-\beta \pi_K)\cdot\sigma^{i}(\zeta) & (\beta-1)\pi_K\cdot\sigma^{i}(\zeta^2) \\ 1-\beta & (\beta-\pi_K)\cdot\sigma^{i}(\zeta) \end{bmatrix}.\]
Similarly, we define $B_i$ $(i=0,\dots, p-1)$ to be the following elements of $PGL_2(L)$:
\[  B_{i}:= \begin{bmatrix} (1-\beta' \pi_K)\cdot\tau^{i}(\alpha) & (\beta'-1)\pi_K\cdot\tau^{i}(\alpha^2) \\ 1-\beta' & (\beta'-\pi_K)\cdot\tau^{i}(\alpha) \end{bmatrix}.\]
Finally, we set 
\[D^+(A_i)=\ol{B}(\sigma^i(\zeta),|\pi_K|^{1+\frac{1}{p}}), D^+(A_i^{-1})=\ol{B}(\pi_K\sigma^i(\zeta),|\beta||\pi_K|^{-(1+\frac{1}{p})}),\] 
and 
\[D^+(B_i)=\ol{B}(\tau^i(\alpha),|\pi_K|^{1+\frac{2}{p}}), D^+(B_i^{-1})=\ol{B}(\pi_K\tau^i(\alpha),|\beta'| |\pi_K|^{-(1+\frac{2}{p})})).\]

\begin{prop}\label{prop:GeneralExamplesDescent}
The subgroup $\Gamma$ of $PGL_2(L)$ generated by $A_0,\dots,
A_{p-1},B_0,\dots,B_{p-1}$ is a Schottky group of rank $2p$ satisfying
all the conditions of Proposition \ref{Pdescentcriterion}.
As a result, the Mumford curve uniformized by $\Gamma$ is defined over $K$.
\end{prop}
\begin{proof}
To prove that $\Gamma$ is a Schottky group, let us show that the $4p$-uple \[\left(D^+(A_i), D^+(B_i), D^+(A_i^{-1}), D^+(B_i)^{-1}\right)_{i=0,\dots,p-1}\] is a Schottky figure associated to $\Gamma$ and adapted to the corresponding set of generators.
First of all, we have that $|\alpha - \tau(\alpha)|=|\pi_K|^{\frac{1}{p}}|1-\zeta^p|=|\pi_K|^{\frac{1}{p}+1}$, and similarly that $|\sigma(\zeta)-\zeta|=|\zeta||1-\zeta^p|=|\pi_K|$, so these discs are disjoint.
Moreover, we can explicitly compute the image of these discs under the associated loxodromic transformation.
For example, by writing $A_0=PAP^{-1}$ with $P=\begin{bmatrix}
-\pi_K\zeta & \zeta \\ -1 & 1
\end{bmatrix}$ and $A=\begin{bmatrix}
\beta & 0 \\ 0 & 1
\end{bmatrix}$, we see that the complement of a disc of the form $\ol{B}(\pi_K\zeta,|\beta|\rho)$ with $|\beta|\rho<1$ and $\rho>1$ is transformed by $A_0$ as follows:

\[ \P^1\setminus \ol{B}(\pi_K\zeta,|\beta|\rho) \mathop{\longrightarrow}^{P^{-1}} B(0,|\beta|^{-1}\rho^{-1}) \mathop{\longrightarrow}^{A} B(0,\rho^{-1}) \mathop{\longrightarrow}^{P}  B(\zeta,\rho^{-1}).\]
In particular, we have that $A_0\big(\P^1\setminus D^+(A_0^{-1})\big) =D^-(A_0)$.
In the same way, by replacing $\zeta$ with $\alpha$, and $\beta$ with $\beta'$, we find that $B_0\big(\P^1\setminus D^+(B_0^{-1})\big) =D^-(B_0)$.
By repeating this computation for the other discs above, we show that
they all satisfy the equation (\ref{eq:Schottkyfig}), and hence they
define a Schottky figure adapted to $(A_0,\dots,A_{p-1},B_0,\dots,
B_{p-1})$.

To prove that this group is fixed under the action of $\Gal(L/K)$ we note that $\sigma(A_i)=A_{i+1}$ for $i=0,\dots, p-2$, $\sigma(A_{p-1})=A_{0}$ and $\tau(A_i)=A_i$ for $i=0,\dots, p-1$, and a similar relation holds for the $B_i$'s.
Hence, $\Gamma$ satisfies condition $(c)$ of Proposition
\ref{Pdescentcriterion}, which implies conditions $(b)$ and $(a)$ of
that proposition.  Condition $(a)$ is what we seek.
\end{proof}

With the Schottky figure found in the proof of Proposition \ref{prop:GeneralExamplesDescent} we can associate a fundamental domain for the action of $\Gamma$:
\[F= \P^{1, \mathrm{an}}_L \setminus \bigcup_{i=0}^{p-1}\left( D^-(A_i) \cup D^-(B_i) \cup D^+(A_i^{-1}) \cup D^+(B_i^{-1}) \right),\]
whose skeleton is depicted in Figure \ref{fig:FunDom}.
The vertices $v_i$ in that figure can be described very concretely as sup-norms of closed discs that we denoted by $\eta_{a,\rho}$ in Remark~\ref{Dc}.
Namely, if we set $|\pi_K|=r$ we have
\[ 
v_1 = \eta_{0,1} \;\; v_2 = \eta_{0,r} \;\; v_3 = \eta_{\zeta,r}\;\; v_4 = \eta_{\pi_K\zeta,r^2}\;\; v_5 = \eta_{\alpha,r^{1+\frac{1}{p}}}\;\; v_6= \eta_{\pi_K\alpha,r^{2+\frac{1}{p}}}.\\
\]

The fundamental domain $F$ has $4p$ boundary points, one for each disc that is a connected component of the complement $\P^{1, \mathrm{an}}_L \setminus F$. In the figure, we have labeled these boundary points with the corresponding discs.
\begin{figure}[h]
\centering
\begin{tikzpicture}[
  level distance=3cm]
  \coordinate (a) at (-2,2.5);
  \coordinate (b) at (0,3);
   \coordinate (c) at (2, 2.5);
   \coordinate (d) at (-2, -1.5);
   \coordinate (e) at (0, -1);
   \coordinate (f) at (2, -1.5);
\draw[thin] (b) -- (e);
\draw[thin] (a) -- (b);
\draw[thin] (b) -- (c);
\draw[thin] (d) -- (e);
\draw[thin] (e) -- (f);
\fill (a) circle[radius=2pt];
\fill (b) circle[radius=2pt];
\fill (c) circle[radius=2pt];
\fill (d) circle[radius=2pt];
\fill (e) circle[radius=2pt];
\fill (f) circle[radius=2pt];
\draw[thin] (a) -- (-4,2.5);
\draw[thin] (a) -- (-4,1.5);
\draw[thin, dashed] (a) -- (-4,2);
\draw[thin] (c) -- (4,2.5);
\draw[thin] (c) -- (4,1.5);
\draw[thin, dashed] (c) -- (4,2);
\draw[thin] (d) -- (-3.95,-0.55);
\draw[thin] (d) -- (-3.93,-1.5);
\draw[thin, dashed] (d) -- (-4,-1);
\draw[thin] (f) -- (3.95,-0.55);
\draw[thin] (f) -- (3.93,-1.5);
\draw[thin, dashed] (f) -- (4,-1);

\node[anchor=south] at (a) {$v_3$};
\node[anchor=south] at (b) {$v_1$};
\node[anchor=south] at (c) {$v_5$};
\node[anchor=north] at (d) {$v_4$};
\node[anchor=north] at (e) {$v_2$};
\node[anchor=north] at (f) {$v_6$};

\fill (-4,2.5) circle[radius=2pt];
\fill (-4,1.5) circle[radius=2pt];
\fill (4,2.5) circle[radius=2pt];
\fill (4,1.5) circle[radius=2pt];
\draw (4,-1.5) circle[radius=2pt];
\draw (4,-0.5) circle[radius=2pt];
\draw (-4,-1.5) circle[radius=2pt];
\draw (-4,-0.5) circle[radius=2pt];

\node[anchor=east] at (-4,2.5) {$D(A_0)$};
\node[anchor=east] at (-4,1.5) {$D(A_{p-1})$};
\node[anchor=west] at (4,2.5) {$D(B_0)$};
\node[anchor=west] at (4,1.5) {$D(B_{p-1})$};
\node[anchor=east] at (-4,-1.5) {$D(A_0^{-1})$};
\node[anchor=east] at (-4,-0.5) {$D(A_{p-1}^{-1})$};
\node[anchor=west] at (4,-1.5) {$D(B_0^{-1})$};
\node[anchor=west] at (4,-0.5) {$D(B_{p-1}^{-1})$};
\end{tikzpicture}
\caption{The skeleton $\Sigma_F$ of the fundamental domain $F$}
\label{fig:FunDom}
\end{figure}
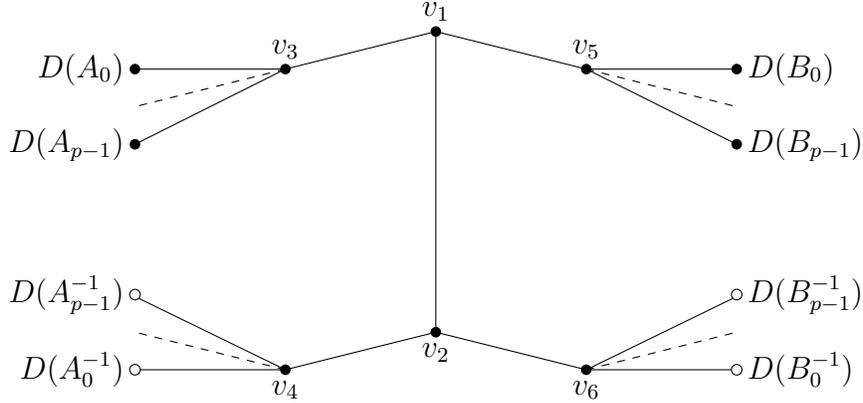

Let us call $Y$ the Mumford curve uniformized by $\Gamma$.
By the second part of Theorem \ref{thm:unif}, the skeleton $\Sigma_Y$ of $Y^{\mathrm{an}}$ is obtained by pairwise identifying the ends of the skeleton of of the fundamental domain $F$, as depicted in Figure \ref{fig:SkeletonEx}.
Since $Y$ has stable reduction, we have that $\Sigma_Y$ is also the dual graph of the special fiber $\mc{Y}^{\stab}_k$ of the stable model $\mc{Y}^{\stab}$ of $Y$.
We deduce that $\mc{Y}^{\stab}_k$ consists of six irreducible components.
For every vertex $v_i$ of such a graph, we denote by $V_i$ the corresponding irreducible component of $\mc{Y}^{\stab}_k$.
\begin{figure}[h]
\centering
\begin{tikzpicture}[
  level distance=3cm]
  \coordinate (a) at (-2,2);
  \coordinate (b) at (0,3);
   \coordinate (c) at (2, 2);
   \coordinate (d) at (-2, -1);
   \coordinate (e) at (0, -2);
   \coordinate (f) at (2, -1);
\draw[thin] (b) -- (e);
\draw[thin] (a) to[out=90,in=90] (b);
\draw[thin] (b) to[out=90,in=90] (c);
\draw[thin] (d) to[out=-90,in=-90] (e);
\draw[thin] (e) to[out=-90,in=-90] (f);
\draw[thin, dashed] (a) -- (d);
\draw[thin] (a) to[out=-120,in=120] (d);
\draw[thin] (a) to[out=-60,in=60] (d);
\draw[thin, dashed] (c) -- (f);
\draw[thin] (c) to[out=-120,in=120] (f);
\draw[thin] (c) to[out=-60,in=60] (f);
\fill (a) circle[radius=2pt];
\fill (b) circle[radius=2pt];
\fill (c) circle[radius=2pt];
\fill (d) circle[radius=2pt];
\fill (e) circle[radius=2pt];
\fill (f) circle[radius=2pt];
\node[rotate=90] at (-3,0) {$p$ edges};
\node[rotate=-90] at (3,0) {$p$ edges};
\node[anchor=east] at (a) {$v_3$};
\node[anchor=east] at (b) {$v_1$};
\node[anchor=east] at (c) {$v_5$};
\node[anchor=east] at (d) {$v_4$};
\node[anchor=east] at (e) {$v_2$};
\node[anchor=east] at (f) {$v_6$};
\end{tikzpicture}
\caption{The dual graph of $\mc{Y}^{\stab}_k$}
\label{fig:SkeletonEx}
\end{figure}

Since the points $v_i$ are fixed by the action of $\Gal(L/K)$ on $\panL$, the irreducible components $V_i$ are also fixed by the action of $\Gal(L/K)$ on $\mc{Y}^{\stab}_k$.
This action cyclically permutes the double points of the intersections $V_3 \cap V_4$ and $V_5 \cap V_6$. 

\begin{remark}\label{rmk:arithmeticform}
Proposition \ref{prop:GeneralExamplesDescent} shows that there is a $K$-curve $X$ whose base change to $L$ is isomorphic to $Y$. 
By construction, the $K$-form $X$ of $Y$ is arithmetic in the sense of \S\ref{Sdescent}.
The fundamental domain $F$ is especially nice, since $g(F)=F$ for every $g\in \Gal(L/K)$.
\end{remark}

\begin{lemma}
The curve $X$ has a $K$-rational point, and the extension $L/K$ is the minimal one over which $X$ acquires semi-stable reduction.
\end{lemma}
\begin{proof}
By Remark \ref{rmk:arithmeticform}, it is sufficient to find a point of $F$ which is fixed by the arithmetic action of $\Gal(L/K)$.
It is easy for instance to see that $0\in F$, which then descends to a
$\Gal(L/K)$-fixed point of $Y$, giving rise to a rational point of
$X$.

The second part of the proof is analogous to that of Proposition~\ref{Pminfield}: if $X$ attained semi-stable reduction over a Galois subextension $K \subseteq M \subsetneq L$, then the $\Gal(L/K)$ action on the special fiber $\mc{Y}^{\stab}_k$ of the stable model $\mc{Y}^{\stab}$ of $Y$ would factor through $\Gal(M/K)$, so we have to show that $\Gal(L/K)$ acts faithfully on $\mc{Y}^{\stab}_k$.
To prove this, note that $\sigma$ acts nontrivially on the points of $V_3\cap V_4$, that are in turn fixed by $\tau$.
Similarly, $\tau$ acts nontrivially on the points of $V_5\cap V_6$, that are in turn fixed by $\tau$.
Then, every element of $\Gal(L/K)$ acts nontrivially on at least one closed point of $\mc{Y}^{\stab}$.
\end{proof}

Let us now define a model $\mc{X}$ of $X$ that we can use to apply the results of \S \ref{Sgeneralities}.
As in that section, we let $\mc{Y}'$ be the semi-stable model of $Y$ associated with the set $S_Y$ of multiplicity 1 points in the skeleton of $Y$.
Then, we define $\mc{Y}' \to \mc{Y}''$ to be the blow down obtained by contracting the components with minimal field of definition equal to $L$.
From the labeling above, we deduce that none of the components $V_i$ for $i=1,\dots, 6$ gets contracted in this way, and hence $\mc{Y}''$ is an admissible blow-up of $\mc{Y}^{\stab}$, and in particular it is again a semi-stable model of $Y$.
As in \S \ref{Sgeneralities}, we set $\mc{X}= \mc{Y}'' / \Gal(L/K)$.

\begin{lemma}[cf.\ Lemma~\ref{Lallcomponents}]\label{lem:generalXmult}
All components of $\mc{X}_k$ have multiplicity dividing $p$.
\end{lemma}
\begin{proof}
This follows from Lemma \ref{lem:multiplicities}, as $\mc{Y}''$ contains only components whose minimal field of definition is strictly contained in $L$.
\end{proof}

\begin{prop}[cf.\ Proposition~\ref{Pweakwild}]\label{prop:generalXmult}
Any principal component $E$ of $\mc{X}_k^{\reg}$ that is exceptional in the desingularization of the model $\mc{X}$ has multiplicity dividing $p$.
\end{prop}
\begin{proof}
Let $x \in \mc{X}$ be the image of $E$ under the desingularization $\mc{X}^{\reg} \to \mc{X}$.
Since the fundamental domain $F$ is fixed by the action of $\Gal(L/K)$ on $\proj^{1, \mathrm{an}}_{L}$, if $\tilde{y} \in \mc{D}''$ lies above $x$ as in
Lemma~\ref{LtransfertoY}, then $\Gal(L/K)\tilde{y}$ lies in the
reduction $\ol{F}$ of $F$.  So by
Proposition~\ref{Pexceptionalmultiplicity1} and
Lemma~\ref{lem:generalXmult}, the multiplicity of $E$ divides $p$.
\end{proof}

\begin{corollary}[cf.\ Corollary~\ref{Cexceptional}]\label{cor:exceptional}
Every principal component of $\mc{X}_k^{\reg}$ has multiplicity dividing $p$.
\end{corollary}
\begin{proof}
This is a direct consequence of Lemma \ref{lem:generalXmult} and Proposition \ref{prop:generalXmult}.
\end{proof}

\section{Divisibility of the stabilization index}\label{Sdivisibility}

In this section, we adopt notation analogous to that of \S\ref{Sexample}.
Specifically, let $X$ be a
curve over $K$ of genus $\geq 1$ with potentially multiplicative reduction, realized over a
minimal Galois extension $L/K$.   Write $Y = X \times_K L$.  Then $Y$ is a Mumford curve with analytification
$\Gamma \backslash \mf{D}$, where $\Gamma$ is the associated Schottky
group and $\mf{D} = \proj^{1,\an}_L \setminus \mc{L}$ with $\mc{L}$ the
limit set of $\Gamma$.  Let $\mc{Y}'$ be the model of $Y$
associated to the full set of multiplicity $1$ points in the skeleton
of $\mf{D}$, as in Proposition~\ref{prop:mult1}.
The morphism $\varpi \colon \mc{D}'
\to \mc{Y}'$ is the morphism of formal schemes coming from the
uniformization map $\mf{D} \to Y^{\an}$ as in Lemma~\ref{lem:formallift}. 
Note that the $L/K$-semilinear action of $\Gal(L/K)$ on $Y$ extends
to $\mc{Y}'$, because $\mc{Y}'$ is defined
canonically in terms of $Y$.  If $g \in \Gal(L/K)$, write $\sigma_g$ for the corresponding automorphism
of $\mc{Y}'$.

Now, let $\mc{X} = \mc{Y}' / \Gal(L/K)$, write $\mc{X}_k$
for its special fiber, and let $\mc{X}^{\reg} \to \mc{X}$ be its minimal regular
snc-resolution.  The special fiber of $\mc{X}^{\reg}$ is denoted $\mc{X}_k^{\reg}$.

%

Define $e(\mc{X}^{\reg})$ to be the $\lcm$ of the multiplicities of the principal
components of $\mc{X}_k^{\reg}$. By \cite[top of p.\
46]{HalleNicaise}, $e(X) \mid e(\mc{X}^{\reg})$. The goal of this
section is to prove Theorem~\ref{Tmain}, by showing that $e(\mc{X}^{\reg}) \mid [L:K]$, which implies $e(X) \mid [L:K]$.  We start with some preparatory results.

\begin{lemma}\label{Ldescent}
Every $L/K$-semilinear automorphism of $Y$ lifts to
an $L/K$-semilinear automorphism of $\mf{D}$,
and any such lift in turn extends to a unique $L/K$-semilinear automorphism on both
$\proj^1_L$ and on $\mc{D}'$.
\end{lemma}

\begin{proof}
The first assertion follows from Proposition~\ref{prop:generallifting}
and (\ref{diagram:unicover}).  The second assertion follows from
Proposition~\ref{prop:extensionofliftings} and the fact that $\mc{D}'$
is canonically constructed from $Y$.  
\end{proof}

\begin{remark}
Compare Lemma~\ref{Ldescent} above with Remark~\ref{Rmumforddescent},
where the goal is to lift an entire $L/K$-semilinear group action, not
just a single automorphism. This is not always possible, but we
give a sufficient criterion for the lifting of such an action in the following proposition.
\end{remark}

\begin{prop}\label{Lfixedcomponent}
  Suppose there is a closed point $y$ on the special fiber
   of $\mc{Y}'$
that is preserved by the action of $\Gal(L/K)$.  Then the canonical
$\Gal(L/K)$-action on $\mc{Y}'$ lifts
to an $L/K$-semilinear $\Gal(L/K)$-action on $\mc{D}'$
as well as on $\proj^1_L$.  Furthermore, the lift can be chosen to
preserve an arbritrary
point of $\mc{D}'$ above $y$.

\end{prop}

\begin{proof}
By Proposition~\ref{prop:extensionofliftings}, proving the first
statement reduces to showing
the $\Gal(L/K)$-action on $\mc{Y}'$ lifts to $\mc{D}'$.  Pick a point $\tilde{y}$ of
$\mc{D}'$ lying above $y$.  For every $\sigma \in \Gal(L/K)$,
choose a lift $\tilde{\sigma}$ of the action of $\sigma$ to $\mc{D}'$
as in Lemma~\ref{Ldescent}. After composing $\tilde{\sigma}$ with an element
of $\Gamma$, we may assume that $\tilde{\sigma}$ fixes $\tilde{y}$.
Indeed, since $\Gamma$ acts freely on $\mc{D}'$, this property
determines $\tilde{\sigma}$ uniquely.  

Now, for all $\sigma, \tau \in
\Gal(L/K)$, we have  
$\tilde{\sigma}\tilde{\tau}$ preserves $\tilde{y}$.  Since
$\tilde{\sigma} \tilde{\tau}$ is a lift of $\sigma \tau$, the uniqueness of
lifts preserving $\tilde{y}$ implies that $\tilde{\sigma}\tilde{\tau} =
\widetilde{\sigma\tau}$.  Thus the entire $\Gal(L/K)$-action lifts to
$\mc{D}'$, proving the first statement.  The point $\tilde{y}$ is fixed by the lift, proving the
second statement.
\end{proof}

\begin{lemma}\label{Lhilb90}
Every $L/K$-semilinear action on $\proj^1_L$ is, up to a change of
coordinate, the purely arithmetic action given by fixing the
coordinate and acting on the coefficients.
\end{lemma}

\begin{proof}
This is essentialy the proof of \cite[Corollary 4.11]{ObusWewers}; 
we reproduce it here and slightly correct it.  Consider, equivalently,
an $L/K$-semilinear action $\rho$ of $\Gal(L/K)$ on the function field
$L(x)$. It is represented by a cocycle in $H^1(\Gal(L/K), PGL_2(L))$, where $\sigma
\in \Gal(L/K)$ is sent to $\alpha \in PGL_2(L)$ such that
$\rho(\sigma)(x) = \alpha(x)$.  By Hilbert's
Theorem 90 (see, e.g., \cite[X, Proposition 3]{Serre79}), this
cohomology set injects into $H^2(G, L^{\times})$, which is trivial by
\cite[Corollary and Example (c) on p.\ 80]{Serre97}.  So the
action is given by a coboundary, and thus has the form $\sigma(x) =
xB^\sigma B^{-1}$ where $B \in PGL_2(L)$ is independent of $\sigma$.  Letting $y
= xB^{-1}$, we see that $g(y) = x(B^\sigma)^{-1}B^\sigma B^{-1} = y$ for all $\sigma
\in \Gal(L/K)$.  This proves the lemma.  
\end{proof}

Recall that we have the following diagram, where $p_1$ is the quotient
map and $p_2$ is the resolution:
   $$
   \xymatrixrowsep{0.8cm}
   \xymatrixcolsep{2cm}
    \xymatrix{
   \   & \mc{Y}' \ar[d]^{p_1} \\
      \mc{X}^{\reg} \ar[r]^-{p_2} & \mc{X} = \mc{Y}'/ \Gal(L/K)
    }
    $$

Since $\mc{Y}'$ has reduced special fiber, all irreducible components of the special fiber of $\mc{X}$ have
multiplicity dividing $[L:K]$.  The same is then true of their strict
transforms in $\mc{X}^{\reg}$.  Consequently, to prove Theorem~\ref{Tmain}, it suffices to show the
following result:

\begin{theorem}\label{Tatx}
Let $x$ be a closed point of $\mc{X}$.  Then every
principal irreducible component of the special fiber
$\mc{X}_k^{\reg}$ of $\mc{X}^{\reg}$ lying in
$p_2^{-1}(x)$ has multiplicity dividing $[L:K]$.
\end{theorem}

In other words, we
can restrict attention to the exceptional divisors contracting to a closed point $x \in \mc{X}$.  For the rest of
\S\ref{Sdivisibility}, fix such a point $x$.

\begin{prop}\label{Petale}
  In order to prove Theorem~\ref{Tatx}, it suffices to assume that
  $p_1^{-1}(x)$ consists of a single point.  That is, 
  that $\Gal(L/K)$ acts on $p_1^{-1}(x)$ with full inertia. 
\end{prop}

\begin{proof}
Pick a point $y \in p_1^{-1}(x)$, let $H \subseteq \Gal(L/K)$ be
the subgroup fixing $y$, and let $K' = L^H$.  Let $p_1'$ be the quotient
morphism $\mc{Y}' \to \mc{Y}' / H =: \mc{Z}$, let $z = p_1'(y)$, and
let $Z = X \times_K K'$.  Then $\mc{Z}$ is a model of $Z$ and the morphism $p_1$
factors as $$\mc{Y}' \stackrel{p_1'}{\to} \mc{Z} \stackrel{p_1''}{\to}
\mc{X}.$$  Since $(p_1')^{-1}(z)$ equals the singleton $\{y\}$ by construction, we may assume that Theorem~\ref{Tatx} holds for $(\mc{Z}, z)$.  We must prove Theorem~\ref{Tatx}
for $(\mc{X}, x)$.

By assumption, the 
morphism $p_1''$ is
\'{e}tale at $z$. 
By \'{e}tale base
change, the map $\mc{Z}^{\reg} := {\mc{X}^{\reg} \times_{\mc{X}} \mc{Z}}
\stackrel{r}{\to} \mc{Z}$ induced from $p_2$ gives a resolution of the singularity
at $z$, and $\mc{Z}^{\reg} \to \mc{X}^{\reg}$ is \'etale along $r^{-1}(z)$.  
Suppose $U$ is a principal irreducible component of  $\mc{X}^{\reg}_k$
lying in $p_2^{-1}(x)$, and $V$ is an irreducible component of $r^{-1}(z)$
lying above $U$.  Note that $V$ is principal
as well.  Since the restriction $V \to
U$ of $\mc{Z}^{\reg} \to \mc{X}^{\reg}$ is generically \'etale, the multiplicity
$m_U$ (which is calculated with respect to the uniformizer $\pi_K$)
is $[K':K]$ times the multiplicity
$m_V$ (which is calculated with respect to the uniformizer $\pi_{K'}$).  By
Theorem~\ref{Tatx} applied to $(\mc{Z},z)$, we have $m_V \mid [L:K']$.  
Thus $m_U$ divides $[L:K'][K':K] = [L : K]$. 
\end{proof}

\begin{proof}[Proof of Theorem~\ref{Tatx}]
  
  After applying Proposition~\ref{Petale}, Lemma~\ref{Lfixedcomponent}
  applies, taking $y$ to be the single point $p_1^{-1}(x)$.  So the action of $\Gal(L/K)$ on
  $\mc{Y}'$ lifts to an action on
  $\mc{D}'$ and on $\proj^1_L$.  By
  Lemma~\ref{Lhilb90}, a coordinate on $\proj^1_L$ can be chosen so
  that $\Gal(L/K)$ fixes this coordinate.

  By Lemma~\ref{Lfixedcomponent}, we may assume there is a closed
  point $\tilde{y} \in \mc{D}'$ above $y$ that is 
  fixed by the lift of the $\Gal(L/K)$-action.  
  Let $\Sigma$ be the set of irreducible components of $\mc{D}'$
  passing through $\tilde{y}$, and let $\mc{D}_{\Sigma}$ and $\phi$ be the model of $\proj^1_L$
  associated to $\Sigma$ and the injection on closed points as in Remark~\ref{Lsamelocalring}.
  The same remark gives us that $\hat{\mc{O}}_{\mc{D}', \tilde{y}} \cong
  \hat{\mc{O}}_{\mc{D}_{\Sigma}, \phi(\tilde{y})}$.

  Furthermore, since $\Gal(L/K)$ fixes
  $\tilde{y}$, it also acts on $\mc{D}_{\Sigma}$, and the Galois action on
  $\hat{\mc{O}}_{\mc{D}, \tilde{y}}$ is isomorphic to that on
  $\hat{\mc{O}}_{\mc{D}_{\Sigma}, \phi(\tilde{y})}$.  Applying Corollary~\ref{Carithmeticaction} to $\mc{D}_{\Sigma}$, all principal components
  of the exceptional fiber of the minimal regular resolution of the singularity of
  $\mc{D}_{\Sigma}/\Gal(L/K)$ at the image of $\phi(\tilde{y})$ under
  the quotient
  have multiplicity dividing $[L:K]$.  Thus the same is true for
  $\mc{D}'/\Gal(L/K)$ at the image of $\tilde{y}$, and thus in turn for
  $\mc{Y}'/(\Gal(L/K))$ at $x$.  This proves the proposition.
\end{proof}

\newpage

\bibliographystyle{alpha}
\bibliography{stabilityindex}

\end{document}